\documentclass{amsart}
\usepackage[utf8]{inputenc}
\usepackage[T1]{fontenc}
\usepackage[english]{babel}
\usepackage{amsthm}
\usepackage{amssymb}
\usepackage{amsmath}
\usepackage{empheq}
\usepackage{graphicx}
\usepackage[export]{adjustbox}
\usepackage[all,knot]{xy}
\usepackage{tikz}
\usetikzlibrary{shapes.geometric,patterns,decorations.pathreplacing}
\usepackage{tikz-cd}
\usepackage{enumitem}
\usepackage{fullpage}

\theoremstyle{definition}
\newtheorem{defi}{Definition}[section]

\theoremstyle{plain}
\newtheorem{thm}[defi]{Theorem}

\newtheorem{lemma}[defi]{Lemma}

\newtheorem{cor}[defi]{Corollary}

\theoremstyle{remark}
\newtheorem{ex}[defi]{Example}

\newtheorem{rem}[defi]{Remark}

\newtheorem*{fact}{Fact}

\newcommand{\scaleValue}{0.35}
\newcommand{\minRadius}{0.75cm}

\newcommand{\ba}[1]{\begin{equation*} \begin{array}{#1}}
\newcommand{\ea}{\end{array} \end{equation*}} 
\newcommand{\noi}{\noindent}
\newcommand{\sk}{\smallskip}
\newcommand{\mbf}{\mathbf}
\newcommand{\mbb}{\mathbb}
\newcommand{\mfk}{\mathfrak}
\newcommand{\mcl}{\mathcal}
\newcommand{\ra}{\rightarrow}
\newcommand{\lra}{\longrightarrow}

\newcommand{\lma}{\longmapsto}

\newcommand{\stars}{\includegraphics[scale=0.75,valign=c]{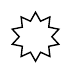} }

\newcommand{\icgDec}{\includegraphics[scale=0.3,valign=c]}

\newcommand{\LieOp}{\mathbf{Lie}}
\newcommand{\hoLie}{\mbf{HoLie}}
\newcommand{\oFun}{\mathcal{O}}
\newcommand{\assOp}{\mathbf{Ass}}
\newcommand{\hoAss}{\assOp_\infty}
\newcommand{\edom}{\mathbf{End}}
\newcommand{\oLie}[1]{\oFun(\LieOp_{#1})}
\newcommand{\oLieC}[1]{\oFun_c(\LieOp_{#1})}
\newcommand{\gra}{\mbf{Gra}}
\newcommand{\deform}{\mbf{Def}}
\newcommand{\defOLie}[1]{\deform(\LieOp_#1 \ra \oLie{#1})}

\newcommand{\fgc}{\mbf{fcGC}}
\newcommand{\grt}{\mfk{grt}}
\newcommand{\gc}{\mbf{GC}}
\newcommand{\GRT}{\mbf{GRT_1}}

\title{From the Lie operad to the Grothendieck-Teichmüller group}
\author{Vincent \textsc{Wolff}}
\address{Mathematics Research Unit, University of Luxembourg, Maison du Nombre, 6 Avenue de la Fonte,
 L-4364 Esch-sur-Alzette, Grand Duchy of Luxembourg }
\email{vincent.wolff@uni.lu}

\begin{document}

\begin{abstract}
We study the deformation complex of the standard morphism from the degree $d$ shifted Lie operad to its polydifferential version,
and prove that it is quasi-isomorphic to the Kontsevich graph complex $\gc_d$. In particular, we show that in the case $d=2$
the Grothendieck-Teichmüller  group $\GRT$ is a symmetry group (up to homotopy) of the aforementioned morphism. We also prove that in the case $d=1$ corresponding to the usual Lie algebras the standard morphism admits a unique homotopy non-trivial deformation which is described explicitly with the help of the universal enveloping construction. Finally we prove the rigidity of the strongly homotopy version of the universal enveloping functor in the Lie theory.
\end{abstract}

\maketitle


\begin{section}{Introduction}

The theory of operads, props and graph complexes undergoes a rapid development in recent years; its applications can be seen nowadays almost everywhere \cite{MR1898414,MR3966807,MR2954392}: in algebraic topology, in homological algebra, in (algebraic) geometry, in string topology, in deformation theory, in quantization theory etc.
This theory demonstrates a remarkable unity of mathematics, relating concepts and results obtained in different branches.


There is a remarkable polydifferential functor $\oFun$ from the category of props to the category of operads,


\ba{rccc}
\oFun:    &\mbf{Props} &\lra & \mbf{Operads} \\
& P & \lma & \oFun(P),
\ea


 \noi whose main defining property is the following one: given any representation of the prop $P$ in a differential graded vector space $V$, there is an associated representation of the operad $\oFun(P)$ on the symmetric tensor algebra $\odot{V}$ given in terms of the polydifferential operators with respect to the standard graded commutative product in $\odot{V}$. It was introduced in \cite{merkulov2015props} to study ribbon graph complexes in the theory of moduli spaces of algebraic curves and it was used, e.g. in \cite{MR4436207} to classify all homotopy classes of M. Kontsevich formality maps and in \cite{MR4179596} to study universal quantizations of Lie bialgebras. 


This paper is devoted to the study of the operad $\oFun(\LieOp_d)$ obtained by applying that functor $\oFun$ to the prop closure of the classical operad $\LieOp_d$ controlling degree $d$ shifted Lie algebras (the case $d=1$ corresponding to usual Lie algebras, often abbreviated  $\LieOp:=\LieOp_1$). This operad can be described in terms of graphs with two types of vertices and, as we discuss below, occurs naturally in the study of the well-known universal enveloping functor $\mfk{U}$ from the category of Lie algebras to the category of associative algebras. This functor can be understood as a morphism of operads 


\begin{equation}
\label{eqGuttMorphism}
\mcl{U}:\assOp \ra \oFun(\LieOp),
\end{equation}

\noi satisfying some natural non-triviality condition (see § \ref{guttQuant} below).


This operad comes equipped with an obvious morphism of operads


\begin{equation}
\label{eqnLie2OLie}
 i:\LieOp_d \ra \oFun_c(\LieOp_d)
\end{equation}


 \noi where $\oFun_c(\LieOp_d)$ is the suboperad of $\oLie{d}$ spanned by connected graphs.


The first main (and perhaps, very surprising) result of this paper says that the deformation complex $\deform(\LieOp_d \stackrel{i}{\ra} \oFun_c(\LieOp_d))$ of the above obvious morphism can be identified (via an explicit quasi-isomorphism shown in § \ref{KonGraphComp} ) with the famous M. Kontsevich's graph complex $\gc_d$! In particular, this result implies that the mysterious Grothendieck-Teichmüller group acts non-trivially and almost faithfully on homotopy classes of the map $i$.


This result gives us one of the simplest incarnations of the Grothendieck-Teichmüller group, and explains, perhaps, why it occurs in two seemingly different deformation quantization problems,
the universal quantization of Poisson structures (solved by M. Kontsevich in \cite{MR2062626} ) and the universal quantization of Lie bialgebras (formulated by V. Drinfeld and solved by Etingov-Kazhdan in \cite{MR1403351}) as both these theories involve $\oFun(\LieOp)$ as a sub-structure. 


The universal enveloping construction associates to a Lie algebra $V$ the corresponding associative algebra $\mfk{U}(V)$ which is isomorphic by the Poincaré-Birkhoff-Witt theorem to the symmetric tensor algebra $\odot V$. Hence it is not that surprising that this construction can be described in terms of the operad $\oFun(\LieOp)$. Indeed, following S. Gutt \cite{MR2795152} and K. Vinay \cite{vinayKontsevich}, one first interprets the universal enveloping algebra construction $\mfk{U}(V)$ as a star product construction on $\odot V$ given in terms of polydifferential operators. Second, one notices that this interpretation can be encoded as a morphism of operads $\assOp \ra \edom_{\odot V}$ satisfying some obvious non-triviality condition. The point is that this map factors through some morphism $\assOp \ra \oFun(\LieOp)$ and the canonical map $\oFun(\LieOp) \ra \edom_{\odot V}$ induced by the functor $\oFun$ applied to the given representation of $\LieOp$ in $V$. Hence all subtleties of the universal enveloping functor $\mfk{U}$ get encoded into the morphism of operads $\assOp \ra \oFun(\LieOp)$ satisfying some non-triviality condition (see § \ref{guttQuant} below).


We study in this paper the deformation complex $\deform(\assOp \stackrel{\mcl{U}}{\ra} \oFun(\LieOp))$ and prove that it is one-dimensional, the unique cohomology class corresponding to the rescaling freedom of the Lie brackets. The conclusion is that the morphism $\mcl{U}$ is unique (up to gauge equivalence). This result is not surprising, of course. One can infer it  from the classification theorem of Kontsevich formality maps given in terms of the graph complex $\gc_2$ \cite{MR4234594}. So we just show a new very short proof of this uniqueness. This proof has a small advantage that it carries over $\LieOp_\infty$ straightforwardly.


There are several attempts to generalize the universal enveloping construction from Lie algebras (controlled by the operad $\LieOp$) to strongly homotopy Lie algebras (which are controlled by $\LieOp_\infty$, the minimal resolution of $\LieOp$). All attempts involve the notion of strongly homotopy associative algebras, which was introduced by J. Stasheff in \cite{MR0158400} and which are controlled by the dg operad $\assOp_\infty$, the minimal resolution of $\assOp$. One such generalization of the functor $\mfk{U}$ is offered by M. Kontsevich formality map applied to linear polyvector fields; that generalization uses, in general, graphs with wheels, but, as has been proven by B. Shoikhet \cite{MR1854132}, the graphs with wheels can be removed so that one gets a strongly homotopy extension of the functor $\mfk{U}$ which works well for infinite dimensional $\LieOp_\infty$ algebras. Another construction was given by V. Baranovsky \cite{MR2470385} as the cobar construction of the Cartan-Chevalley-Eilenberg coalgebra associated to an $\LieOp_\infty$ algebra, by J. M. Moreno-Fernández  \cite{MR4381197} as an $\hoAss$ algebra  isomorphic as graded vector spaces to the free symmetric algebra associated to an $\LieOp_\infty$ algebra and by T. Lada and M. Markl in \cite{MR1327129}.


Both these constructions can be understood as a morphism of operads $\assOp_\infty \ra \oFun(\LieOp_\infty)$. We study the deformation complex of any of these two maps and prove that it is quasi-isomorphic to the deformation complex of the map considered in $(\ref{eqGuttMorphism})$, implying that both constructions are gauge equivalent. Moreover, any other attempt to construct such a generalization satisfying some natural conditions must be gauge equivalent to these ones as the cohomology of the complex $\deform(\hoAss \ra \oFun(\LieOp_\infty))$ is one-dimensional (see Corollary \ref{cor:GuttInfinity}). This cohomology result is in a full agreement  with the derived Poincaré-Birkhoff-Witt theorem established by A. Khoroshkin and P. Tamaroff in \cite{Khoroshkin2023}.

\end{section}

\section*{Acknowledgments}

I would like to thank my supervisor Sergei Merkulov for suggesting this topic and for many encouraging discussions.

\begin{section}{Reminder on classical operads}

Let us remind the construction of two classical operads \cite{MR2954392,MR1898414}.


Let $E=\{E(n)\}$ be the $\mbb{S}$-module defined by $E(n)=0$ except

\smallskip

\begin{equation*}
E(2)=\text{sgn}_2^d[d-1]=\text{span}\left\langle 
\begin{tikzpicture}[baseline={([yshift=-0.5ex]current bounding box.center)},scale=\scaleValue,transform shape]
\node [circle,draw,fill] (v0) at (0,0) {};
\node (v1) at (-1,-1) {\huge $1$};
\node (v2) at (1,-1) {\huge $2$};
\node (v3) at (0,1) {};
\draw (v0) edge (v1);
\draw (v0) edge (v2);
\draw (v0) edge (v3);
\end{tikzpicture}
=(-1)^d
\begin{tikzpicture}[baseline={([yshift=-0.5ex]current bounding box.center)},scale=\scaleValue,transform shape]
\node [circle,draw,fill] (v0) at (0,0) {};
\node (v1) at (-1,-1) {\huge $2$};
\node (v2) at (1,-1) {\huge $1$};
\node (v3) at (0,1) {};
\draw (v0) edge (v1);
\draw (v0) edge (v2);
\draw (v0) edge (v3);
\end{tikzpicture} \right\rangle. \\
\end{equation*}

\smallskip

The operad of degree $d$ shifted Lie algebras is defined to be the quotient $\LieOp_d:=\text{Free}(E)/I$, where $I$ is generated by the element

\smallskip

\begin{equation*}
\begin{tikzpicture}[baseline={([yshift=-0.5ex]current bounding box.center)},scale=\scaleValue,transform shape]
\node [circle,draw,fill] (v0) at (0,0) {};
\node [circle,draw,fill] (v1) at (-1,-1) {};
\node (v2) at (0,1) {};
\node (v3) at (1,-1) {\huge $3$};
\node (v4) at (-2,-2) {\huge $1$};
\node (v5) at (0,-2) {\huge $2$};
\draw (v0) edge (v1);
\draw (v0) edge (v2);
\draw (v0) edge (v3);
\draw (v1) edge (v4);
\draw (v1) edge (v5);
\end{tikzpicture}
+
\begin{tikzpicture}[baseline={([yshift=-0.5ex]current bounding box.center)},scale=\scaleValue,transform shape]
\node [circle,draw,fill] (v0) at (0,0) {};
\node [circle,draw,fill] (v1) at (-1,-1) {};
\node (v2) at (0,1) {};
\node (v3) at (1,-1) {\huge $2$};
\node (v4) at (-2,-2) {\huge $3$};
\node (v5) at (0,-2) {\huge $1$};
\draw (v0) edge (v1);
\draw (v0) edge (v2);
\draw (v0) edge (v3);
\draw (v1) edge (v4);
\draw (v1) edge (v5);
\end{tikzpicture}
+
\begin{tikzpicture}[baseline={([yshift=-0.5ex]current bounding box.center)},scale=\scaleValue,transform shape]
\node [circle,draw,fill] (v0) at (0,0) {};
\node [circle,draw,fill] (v1) at (-1,-1) {};
\node (v2) at (0,1) {};
\node (v3) at (1,-1) {\huge $1$};
\node (v4) at (-2,-2) {\huge $2$};
\node (v5) at (0,-2) {\huge $3$};
\draw (v0) edge (v1);
\draw (v0) edge (v2);
\draw (v0) edge (v3);
\draw (v1) edge (v4);
\draw (v1) edge (v5);
\end{tikzpicture}.
\end{equation*}


When studying linear combinations of graphs built from this generating corolla, it is useful, not to make sign mistakes, to view each such a corolla as a degree $d$ vertex with two degree $-d$ incoming half-edges and one degree $1$ outgoing half-edge. Hence for $d$ odd, such a graph has vertices of odd degree (and thus an ordering of this set, up to a permutation $\sigma$ and multiplication by $\text{sgn}(\sigma)$, has to be chosen), while internal edges have degree $1-d$. In the case $d$ even, the vertices are even, but the internal edges are odd so that an ordering of this set is chosen (up to permutation). This rule helps us understand what kind of implicit ordering is hidden in a graph like this:


\[
\begin{tikzpicture}[baseline={([yshift=-0.5ex]current bounding box.center)},scale=\scaleValue,transform shape]
\node (v1) at (0,1) {};
\node [circle,draw,fill] (v2) at (0,0) {};
\node [circle,draw,fill] (v3) at (-1,-1) {};
\node [circle,draw,fill] (v6) at (1,-1) {};
\node  (v4) at (-1.5,-2) {\huge $1$};
\node  (v5) at (-0.5,-2) {\huge $2$};
\node  (v7) at (0.5,-2) {\huge $3$};
\node  (v8) at (1.5,-2) {\huge $4$};
\draw  (v1) edge (v2);
\draw  (v2) edge (v3);
\draw  (v3) edge (v4);
\draw  (v3) edge (v5);
\draw  (v2) edge (v6);
\draw  (v6) edge (v7);
\draw  (v6) edge (v8);
\end{tikzpicture}.
\]


The minimal resolution of $\LieOp$ is given by the dg free operad $\hoLie_d=(\mathbf{Free}\langle E \rangle,\delta)$ generated by the $\mathbb{S}$-module

\smallskip

\[
E = \{E(n)=(\text{sgn}_n)^{\otimes |d|}[nd-d-1]=\text{span} \left\langle \begin{tikzpicture}[baseline={([yshift=-0.5ex]current bounding box.center)},scale=\scaleValue,transform shape]
\node  [circle,draw,fill] (v1) at (0,0) {};
\node (v2) at (0,1) {};
\node (v4) at (-1,-1) {\huge $2$};
\node (v3) at (-2,-1) {\huge $1$};
\node (v5) at (1,-1) {};
\node (v6) at (2,-1) {\huge $n$};
\node at (0,-1) {\huge $\cdots$};
\draw  (v1) edge (v2);
\draw  (v1) edge (v3);
\draw  (v1) edge (v4);
\draw  (v1) edge (v5);
\draw  (v1) edge (v6);
\end{tikzpicture} \right\rangle \}_{n \geq 2},
\]

\smallskip

\noindent for $n \geq 2$, where 

\begin{equation*}
(\text{sgn}_n)^{\otimes |d|} = \left\{ \begin{array}{ll}
\text{sgn}_n & \text{if $d$ is odd} \\
1_n & \text{if $d$ is even},
\end{array} \right.
\end{equation*}

\smallskip

\noi where $\text{sgn}_n$ denotes the $1$-dimensional sign representation and $1_n$ denotes the $1$-dimensional trivial representation of $\mbb{S}_n$. The differential is given on generators by

\smallskip

\[
\delta \left( \begin{tikzpicture}[baseline={([yshift=-0.5ex]current bounding box.center)},scale=\scaleValue,transform shape]
\node  [circle,draw,fill] (v1) at (0,0) {};
\node (v2) at (0,1) {};
\node (v4) at (-1,-1) {\huge $2$};
\node (v3) at (-2,-1) {\huge $1$};
\node (v5) at (1,-1) {};
\node (v6) at (2,-1) {\huge $n$};
\node at (0,-1) {\huge $\cdots$};
\draw  (v1) edge (v2);
\draw  (v1) edge (v3);
\draw  (v1) edge (v4);
\draw  (v1) edge (v5);
\draw  (v1) edge (v6);
\end{tikzpicture} \right)
= \sum_{\substack{I \sqcup J=[n] \\ |I|,|J| \geq 2}}{(-1)^{\text{sgn}(I,J)+(|I|+1)|J|}\includegraphics[scale=0.5,valign=c]{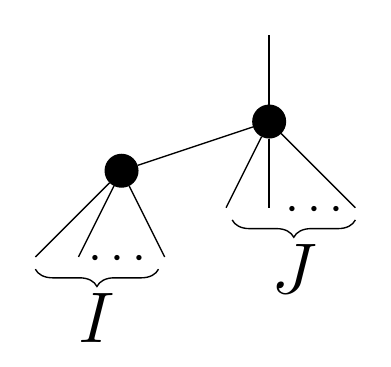}},
\]


\noi where $\text{sgn}(I,J)$ denotes the signature of the permutation $[n] \ra [I,J]$. The case $d=1$ is also denoted by $\LieOp_\infty:=\hoLie_1$.

\sk
\sk
\sk

Let $E=\{E(n)\}$ be the $\mbb{S}$-module defined by $E(n)=0$ except


\[
E(2)=\text{id}_2=\text{span}\left\langle\begin{tikzpicture}[baseline={([yshift=-0.5ex]current bounding box.center)},scale=\scaleValue,transform shape]
\node (v0) at (0,1) {};
\node (v1) [circle,draw] at (0,0) {};
\node (v2) at (-1,-1) {\huge $1$};
\node (v3) at (1,-1) {\huge $2$};
\draw (v0) edge (v1);
\draw (v1) edge (v2);
\draw (v1) edge (v3);
\end{tikzpicture} \right\rangle .
\]


The operad of associative algebras is defined by the quotient $\assOp:=\text{Free}(E)/I$ where $I$ is generated by


\[
\begin{tikzpicture}[baseline={([yshift=-0.5ex]current bounding box.center)},scale=\scaleValue,transform shape]
\node (v0) at (0,1) {};
\node (v1) [circle,draw] at (0,0) {};
\node (v2) [circle,draw] at (-1,-1) {};
\node (v3) at (1,-1) {\huge $3$};
\node (v4) at (-2,-2) {\huge $1$};
\node (v5) at (0,-2) {\huge $2$};
\draw (v0) edge (v1);
\draw (v1) edge (v2);
\draw (v1) edge (v3);
\draw (v2) edge (v4);
\draw (v2) edge (v5);
\end{tikzpicture}
-
\begin{tikzpicture}[baseline={([yshift=-0.5ex]current bounding box.center)},scale=\scaleValue,transform shape]
\node (v0) at (0,1) {};
\node (v1) [circle,draw] at (0,0) {};
\node (v2) [circle,draw] at (1,-1) {};
\node (v3) at (-1,-1) {\huge $1$};
\node (v4) at (0,-2) {\huge $2$};
\node (v5) at (2,-2) {\huge $3$};
\draw (v0) edge (v1);
\draw (v1) edge (v2);
\draw (v1) edge (v3);
\draw (v2) edge (v4);
\draw (v2) edge (v5);
\end{tikzpicture}.
\]


Its minimal resolution is given by the dg free operad $\assOp_\infty:=(\text{Free}(E),\delta)$, where $E$ is the $\mbb{S}$-module generated by


\[
E=\{E(n)=\mbb{K}[\mbb{S}_n][n-2]=\text{span} \left\langle\begin{tikzpicture}[baseline={([yshift=-.5ex]current bounding box.center)},scale=\scaleValue,transform shape]
\node [draw,circle] (v2) at (0,0) {};
\node [label=below:\huge $\tau(1)$] (v1) at (-3,-1) {};
\node [label=below:\huge $\tau(2)$] (v3) at (-1.5,-1) {};
\node (v4) at (1.5,-1) {};
\node [label=below:\huge $\tau(n)$] (v5) at (3,-1) {};
\node (v6) at (0,1) {};
\draw  (v1) edge (v2);
\draw  (v3) edge (v2);
\draw  (v2) edge (v4);
\draw  (v2) edge (v5);
\draw  (v2) edge (v6);
\node at (0.2,-1.5) {\huge $\cdots$};
\end{tikzpicture} \right\rangle_{\tau \in \mbb{S}_n} \}_{n \geq 2}
\]


\noi and the differential is given on generators by


\[
\delta \left( \begin{tikzpicture}[baseline={([yshift=-.5ex]current bounding box.center)},scale=\scaleValue,transform shape]
\node [draw,circle] (v2) at (0,0) {};
\node [label=below:\huge $\tau(1)$] (v1) at (-3,-1) {};
\node [label=below:\huge $\tau(2)$] (v3) at (-1.5,-1) {};
\node (v4) at (1.5,-1) {};
\node [label=below:\huge $\tau(n)$] (v5) at (3,-1) {};
\node (v6) at (0,1) {};
\draw  (v1) edge (v2);
\draw  (v3) edge (v2);
\draw  (v2) edge (v4);
\draw  (v2) edge (v5);
\draw  (v2) edge (v6);
\node at (0.2,-1.5) {\huge $\cdots$};
\end{tikzpicture} \right):=\sum_{k=0}^{n-2}{\sum_{l=2}^{n-k}{(-1)^{k+l(n-k-l)+1} \; \begin{tikzpicture}[baseline={([yshift=-.5ex]current bounding box.center)},scale=\scaleValue,transform shape]
\node [draw,circle] (v1) at (0,0) {};
\node (v2) at (0,1) {};
\node [label=below:\huge $\tau(1)$] (v3) at (-4.5,-1) {};
\node [label=below:\huge $\tau(k)$] (v4) at (-2,-1) {};
\node [draw,circle] (v5) at (0,-2.5) {};
\node [label=below:\huge $\tau(k+l+1)$] (v6) at (3,-1) {};
\node [label=below:\huge $\tau(n)$] (v7) at (6.5,-1) {};
\node [label=below:\huge $\tau(k+1)$] (v8) at (-2,-3.5) {};
\node (v9) at (-1,-3.5) {};
\node (v10) at (1,-3.5) {};
\node [label=below:\huge $\tau(k+l)$] (v11) at (2,-3.5) {};
\draw  (v1) edge (v2);
\draw  (v1) edge (v3);
\draw  (v4) edge (v1);
\draw  (v1) edge (v5);
\draw  (v1) edge (v6);
\draw  (v1) edge (v7);
\draw  (v5) edge (v8);
\draw  (v9) edge (v5);
\draw  (v5) edge (v10);
\draw  (v5) edge (v11);
\node at (-3.2,-1.5) {\huge $\cdots$};
\node at (0,-3.5) {\huge $\cdots$};
\node at (4,-1) {\huge $\cdots$};
\end{tikzpicture}}}.
\]


Representations of $\hoAss$ in a dg vector space $V$ are precisely the strongly homotopy associative algebras introduced by J. Stasheff in \cite{MR0158400}.

As a final remark we recall that there is a morphism of operads $\LieOp \ra \assOp$ given by

\[
\begin{tikzpicture}[baseline={([yshift=-0.5ex]current bounding box.center)},scale=\scaleValue,transform shape]
\node (v0) at (0,1) {};
\node (v1) [circle,draw,fill] at (0,0) {};
\node (v2) at (-1,-1) {\huge $1$};
\node (v3) at (1,-1) {\huge $2$};
\draw (v0) edge (v1);
\draw (v1) edge (v2);
\draw (v1) edge (v3);
\end{tikzpicture}
 \lma 
\frac{1}{2} \left(
\begin{tikzpicture}[baseline={([yshift=-0.5ex]current bounding box.center)},scale=\scaleValue,transform shape]
\node (v0) at (0,1) {};
\node (v1) [circle,draw] at (0,0) {};
\node (v2) at (-1,-1) {\huge $1$};
\node (v3) at (1,-1) {\huge $2$};
\draw (v0) edge (v1);
\draw (v1) edge (v2);
\draw (v1) edge (v3);
\end{tikzpicture}
-\begin{tikzpicture}[baseline={([yshift=-0.5ex]current bounding box.center)},scale=\scaleValue,transform shape]
\node (v0) at (0,1) {};
\node (v1) [circle,draw] at (0,0) {};
\node (v2) at (-1,-1) {\huge $2$};
\node (v3) at (1,-1) {\huge $1$};
\draw (v0) edge (v1);
\draw (v1) edge (v2);
\draw (v1) edge (v3);
\end{tikzpicture} \right).
\]

This map will be used in Section \ref{uniqueDeformation}.

\end{section}

\begin{section}{Polydifferential functor}

We need an endofunctor in the category of dg operads constructed in \cite{merkulov2015props}. Given an augmented prop $\mcl{P}$ there is an operad $\oFun(\mcl{P})$ such that for any representation $\mcl{P} \ra \edom_V$ in a dg vector space $V$ there is a representation $\oFun(\mcl{P}) \ra \edom_{\odot^\bullet(V)}$ in the graded algebra $\odot^\bullet(V)$ where the elements of $\oFun(\mcl{P})$ act as polydifferential operators. When $\mcl{P}$ is an operad we consider its prop closure, which we also denote $\mcl{P}$. The elements of the prop closure are given by disjoint unions of elements of $\mcl{P}$ with inputs and outputs labelled differently. For example


\[
\begin{tikzpicture}[baseline={([yshift=-0.5ex]current bounding box.center)},scale=\scaleValue,transform shape]
\node (v0) at (0,1) {\huge $1$};
\node (v1) [circle,draw,fill] at (0,0) {};
\node (v2) [circle,draw,fill] at (-1,-1) {};
\node (v3) at (-2,-2) {\huge $1$};
\node (v4) at (0,-2) {\huge $2$};
\node (v5) at (1,-1) {\huge $3$};
\draw (v0) edge (v1);
\draw (v1) edge (v2);
\draw (v1) edge (v5);
\draw (v2) edge (v3);
\draw (v2) edge (v4);
\end{tikzpicture}
\, \,
\begin{tikzpicture}[baseline={([yshift=-0.5ex]current bounding box.center)},scale=\scaleValue,transform shape]
\node (v0) at (0,1) {\huge $2$};
\node (v1) [circle,draw,fill] at (0,0) {};
\node (v2) at (-1,-1) {\huge $4$};
\node (v3) at (1,-1) {\huge $5$};
\draw (v0) edge (v1);
\draw (v1) edge (v2);
\draw (v1) edge (v3);
\end{tikzpicture}
\, \,
\in \LieOp_d(2,5).
\]


We are mainly interested in the operad $\oLie{d}$ obtained by applying the functor $\oFun$ to the prop closure of the operad of degree $d$ shifted Lie algebras. Elements of $\oLie{d}$ are obtained joining the outputs of an element of the prop closure of $\LieOp_d$ to a new white vertex and also by partitioning all of the input legs into groups and joining each group them to a new labelled white vertex. For example

\[
\begin{tikzpicture}[baseline={([yshift=-0.5ex]current bounding box.center)},scale=\scaleValue,transform shape]
\node [circle,draw,minimum size = \minRadius] (v1) at (0,1.25) {};
\node [circle,draw,fill] (v2) at (-1,0) {};
\node [circle,draw] (v3) at (0.5,-2) {\huge $3$};
\node [circle,draw,fill] (v4) at (-2,-1) {};
\node [circle,draw] (v5) at (-3,-2) {\huge $1$};
\node [circle,draw] (v6) at (-1,-2) {\huge $2$};
\node [circle,draw,fill] (v7) at (2,0) {};
\node [circle,draw] (v8) at (2,-2) {\huge $4$};
\draw  (v1) edge (v2);
\draw  (v2) edge (v3);
\draw  (v2) edge (v4);
\draw  (v4) edge (v5);
\draw  (v4) edge (v6);
\draw [bend right] (v7) edge (v8);
\draw [bend left] (v7) edge (v8);
\draw [bend left] (v1) edge (v7);
\end{tikzpicture} \in \oFun(\LieOp_d)(4).
\]


The black vertices are called \textit{internal vertices} and edges between internal vertices are called \textit{internal edges}. By erasing all white vertices we obtain a collection of disjoint trees which will be called \textit{internal irreducible components} (i.c.c.).


Note that for $d$ even we implicitly assume an ordering of the internal edges as well as an ordering of the in-edges of the white out-vertex (up to a sign). For $d$ odd, we assume implicitly an ordering of the internal vertices as well as an ordering (up to a sign) of the out-edges attached to each white in-vertex.




The compositions $\Gamma_1 \text{o}_i \Gamma_2$ in $\oLie{d}$ work as follows: first we erase the white out-vertex of $\Gamma_2$ and the $i$th white in-vertex of $\Gamma_1$. This step creates many "hanging" out-edges in $\Gamma_2$ and in-edges in $\Gamma_1$. Second we sum over all possible attachments of the hanging out-edges of $\Gamma_2$ to the hanging in-edges of $\Gamma_1$ and the white out-vertex of $\Gamma_1$. Finally we sum over all attachments of remaining in-edges of $\Gamma_1$ to the white in-vertices of $\Gamma_2$.

\begin{ex}

\begin{enumerate}[label=\arabic*)]

\item

$\begin{tikzpicture}[baseline={([yshift=-0.5ex]current bounding box.center)},scale=\scaleValue,transform shape]
\node [circle,draw,fill] (v1) at (0,0) {};
\node [circle,draw] (v2) at (-1,-1) {\huge $1$};
\node [circle,draw] (v3) at (1,-1) {\huge $2$};
\node [circle,minimum size =\minRadius,draw] (v4) at (0,1.25) {};
\draw  (v1) edge (v2);
\draw  (v1) edge (v3);
\draw  (v1) edge (v4);
\end{tikzpicture}
\text{o}_2 \begin{tikzpicture}[baseline={([yshift=-0.5ex]current bounding box.center)},scale=\scaleValue,transform shape]
\node [circle,draw,fill] (v1) at (0,0) {};
\node [circle,draw] (v2) at (-1,-1) {\huge $1$};
\node [circle,draw] (v3) at (1,-1) {\huge $2$};
\node [circle,minimum size =\minRadius,draw] (v4) at (0,1.25) {};
\draw  (v1) edge (v2);
\draw  (v1) edge (v3);
\draw  (v1) edge (v4);
\end{tikzpicture}
=
\begin{tikzpicture}[baseline={([yshift=-0.5ex]current bounding box.center)},scale=\scaleValue,transform shape]
\node (v0) [circle,draw,minimum size = \minRadius] at (0,1.25) {};
\node (v1) [circle,draw,fill] at (0,0) {};
\node (v2) [circle,draw] at (-1,-1) {\huge $1$};
\node (v3) [circle,draw,fill] at (1,-1) {};
\node (v4) [circle,draw] at (0,-2) {\huge $2$};
\node (v5) [circle,draw] at (2,-2) {\huge $3$};
\draw (v0) edge (v1);
\draw (v1) edge (v2);
\draw (v1) edge (v3);
\draw (v3) edge (v4);
\draw (v3) edge (v5);
\end{tikzpicture}
+
\begin{tikzpicture}[baseline={([yshift=-0.5ex]current bounding box.center)},scale=\scaleValue,transform shape]
\node (v0) [circle,draw,minimum size = \minRadius] at (0,1.25) {};
\node (v1) [circle,draw,fill] at (-1,0) {};
\node (v2) [circle,draw,fill] at (1,0) {};
\node (v3) [circle,draw] at (-2,-1) {\huge $1$};
\node (v4) [circle,draw] at (0,-1) {\huge $2$};
\node (v5) [circle,draw] at (2,-1) {\huge $3$};
\draw (v0) edge (v1);
\draw (v0) edge (v2);
\draw (v1) edge (v3);
\draw (v1) edge (v4);
\draw (v2) edge (v4);
\draw (v2) edge (v5);
\end{tikzpicture}
+
\begin{tikzpicture}[baseline={([yshift=-0.5ex]current bounding box.center)},scale=\scaleValue,transform shape]
\node (v0) [circle,draw,minimum size = \minRadius] at (0,1.25) {};
\node (v1) [circle,draw,fill] at (-1,0) {};
\node (v2) [circle,draw,fill] at (1,0) {};
\node (v3) [circle,draw] at (-2,-1) {\huge $1$};
\node (v4) [circle,draw] at (0,-1) {\huge $2$};
\node (v5) [circle,draw] at (2,-1) {\huge $3$};
\draw (v0) edge (v1);
\draw (v0) edge (v2);
\draw (v1) edge (v3);
\draw (v1) edge (v5);
\draw (v2) edge (v4);
\draw (v2) edge (v5);
\end{tikzpicture}.
$

\item

$
\begin{tikzpicture}[baseline={([yshift=-0.5ex]current bounding box.center)},scale=\scaleValue,transform shape]
\node [circle,draw,minimum size = \minRadius] (v2) at (0,1.25) {};
\node [circle,draw,fill] (v1) at (-1.5,0) {};
\node [circle,draw,fill] (v3) at (1.5,0) {};
\node [circle,draw] (v6) at (-2.5,-1.5) {\Large $1$};
\node [circle,draw] (v7) at (-0.5,-1.5) {\Large $2$};
\node [circle,draw] (v5) at (1.5,-1.5) {\Large $3$};
\draw  (v1) edge (v2);
\draw  (v2) edge (v3);
\draw  (v3) [bend left] edge (v5);
\draw  (v1) edge (v6);
\draw  (v1) edge (v7);
\draw [bend right] (v3) edge (v5);
\end{tikzpicture}
\text{o}_1
\begin{tikzpicture}[baseline={([yshift=-0.5ex]current bounding box.center)},scale=\scaleValue,transform shape]
\node (v0) [circle,draw,minimum size = \minRadius] at (0,1.25) {};
\node (v1) [circle,draw,fill] at (0,0) {};
\node (v2) [circle,draw] at (-1,-1) {\huge $1$};
\node (v3) [circle,draw] at (1,-1) {\huge $2$};
\draw (v0) edge (v1);
\draw (v1) edge (v2);
\draw (v1) edge (v3);
\end{tikzpicture}
=
\begin{tikzpicture}[baseline={([yshift=-0.5ex]current bounding box.center)},scale=\scaleValue,transform shape]
\node [circle,draw,minimum size = \minRadius] (v1) at (0,1.25) {};
\node [circle,draw,fill] (v2) at (-1.5,0) {};
\node [circle,draw,fill] (v3) at (-2.5,-1) {};
\node [circle,draw,fill] (v7) at (1.5,0) {};
\node [circle,draw] (v4) at (-3.5,-2) {\huge $1$};
\node [circle,draw] (v5) at (-1.5,-2) {\huge $2$};
\node [circle,draw] (v6) at (-0.5,-1) {\huge $3$};
\node [circle,draw] (v8) at (1.5,-2) {\huge $4$};
\draw  (v1) edge (v2);
\draw  (v2) edge (v3);
\draw  (v3) edge (v4);
\draw  (v3) edge (v5);
\draw  (v2) edge (v6);
\draw  (v1) edge (v7);
\draw [bend left] (v7) edge (v8);
\draw [bend right] (v7) edge (v8);
\end{tikzpicture}
+
\begin{tikzpicture}[baseline={([yshift=-0.5ex]current bounding box.center)},scale=\scaleValue,transform shape]
\node [circle,draw,minimum size = \minRadius] (v2) at (0,1.25) {};
\node [circle,draw,fill] (v1) at (-1.5,0) {};
\node [circle,draw] (v3) at (-2.5,-2) {\huge $1$};
\node [circle,draw] (v4) at (-0.5,-2) {\huge $2$};
\node [circle,draw,fill] (v5) at (1,0) {};
\node [circle,draw,fill] (v6) at (2.5,0) {};
\node [circle,draw] (v8) at (1,-2) {\huge $3$};
\node [circle,draw] (v7) at (2.5,-2) {\huge $4$};
\draw  (v2) edge (v1);
\draw  (v2) edge (v5);
\draw  (v2) edge (v6);
\draw [bend left] (v6) edge (v7);
\draw [bend right] (v6) edge (v7);
\draw  (v1) edge (v3);
\draw  (v1) edge (v4);
\draw  (v5) edge (v8);
\draw  (v5) edge (v3);
\end{tikzpicture}
+
\begin{tikzpicture}[baseline={([yshift=-0.5ex]current bounding box.center)},scale=\scaleValue,transform shape]
\node [circle,draw,minimum size = \minRadius] (v2) at (0,1.25) {};
\node [circle,draw,fill] (v1) at (-1.5,0) {};
\node [circle,draw] (v3) at (-2.5,-2) {\huge $1$};
\node [circle,draw] (v4) at (-0.5,-2) {\huge $2$};
\node [circle,draw,fill] (v5) at (1,0) {};
\node [circle,draw,fill] (v6) at (2.5,0) {};
\node [circle,draw] (v8) at (1,-2) {\huge $3$};
\node [circle,draw] (v7) at (2.5,-2) {\huge $4$};
\draw  (v2) edge (v1);
\draw  (v2) edge (v5);
\draw  (v2) edge (v6);
\draw [bend left] (v6) edge (v7);
\draw [bend right] (v6) edge (v7);
\draw  (v1) edge (v3);
\draw  (v1) edge (v4);
\draw  (v5) edge (v8);
\draw  (v5) edge (v4);
\end{tikzpicture}.
$
\begin{flushleft}

\end{flushleft}

\end{enumerate}

\end{ex}


We define $\oLieC{d}$ to be the suboperad spanned by connected graphs, i.e. we assume that the graphs remain connected when we erase the white output vertex.


\begin{lemma}

There is a morphism of operads 


\begin{equation}
\label{morphism_lie_olie}
i:\LieOp_d \longrightarrow \oLieC{d}
\end{equation}


\noi given by


\[
\begin{tikzpicture}[baseline={([yshift=-0.5ex]current bounding box.center)},scale=\scaleValue,transform shape]
\node (v0) at (0,1) {};
\node (v1) [circle,draw,fill] at (0,0) {};
\node (v2) at (-1,-1) {\huge $1$};
\node (v3) at (1,-1) {\huge $2$};
\draw (v0) edge (v1);
\draw (v1) edge (v2);
\draw (v1) edge (v3);
\end{tikzpicture}
\longmapsto
\begin{tikzpicture}[baseline={([yshift=-0.5ex]current bounding box.center)},scale=\scaleValue,transform shape]
\node [circle,draw,fill] (v1) at (0,0) {};
\node [circle,draw] (v2) at (-1,-1) {\huge $1$};
\node [circle,draw] (v3) at (1,-1) {\huge $2$};
\node [circle,draw,minimum size = \minRadius] (v4) at (0,1.25) {};
\draw  (v1) edge (v2);
\draw  (v1) edge (v3);
\draw  (v1) edge (v4);
\end{tikzpicture}.
\]

\end{lemma}
\begin{proof}

It suffices to check that the Jacobi identity is mapped to $0$. Indeed we need that 


\ba{rl}
\begin{tikzpicture}[baseline={([yshift=-0.5ex]current bounding box.center)},scale=0.3,transform shape]
\node (v0) [circle,draw,minimum size = \minRadius] at (0,1.25) {};
\node (v1) [circle,draw,fill] at (0,0) {};
\node (v2) [circle,draw] at (-1,-1) {\huge $1$};
\node (v3) [circle,draw,fill] at (1,-1) {};
\node (v4) [circle,draw] at (0,-2) {\huge $2$};
\node (v5) [circle,draw] at (2,-2) {\huge $3$};
\draw (v0) edge (v1);
\draw (v1) edge (v2);
\draw (v1) edge (v3);
\draw (v3) edge (v4);
\draw (v3) edge (v5);
\end{tikzpicture}
+
\begin{tikzpicture}[baseline={([yshift=-0.5ex]current bounding box.center)},scale=0.3,transform shape]
\node (v0) [circle,draw,minimum size = \minRadius] at (0,1.25) {};
\node (v1) [circle,draw,fill] at (-1,0) {};
\node (v2) [circle,draw,fill] at (1,0) {};
\node (v3) [circle,draw] at (-2,-1) {\huge $1$};
\node (v4) [circle,draw] at (0,-1) {\huge $2$};
\node (v5) [circle,draw] at (2,-1) {\huge $3$};
\draw (v0) edge (v1);
\draw (v0) edge (v2);
\draw (v1) edge (v3);
\draw (v1) edge (v4);
\draw (v2) edge (v4);
\draw (v2) edge (v5);
\end{tikzpicture}
+
\begin{tikzpicture}[baseline={([yshift=-0.5ex]current bounding box.center)},scale=0.3,transform shape]
\node (v0) [circle,draw,minimum size = \minRadius] at (0,1.25) {};
\node (v1) [circle,draw,fill] at (-1,0) {};
\node (v2) [circle,draw,fill] at (1,0) {};
\node (v3) [circle,draw] at (-2,-1) {\huge $1$};
\node (v4) [circle,draw] at (0,-1) {\huge $2$};
\node (v5) [circle,draw] at (2,-1) {\huge $3$};
\draw (v0) edge (v1);
\draw (v0) edge (v2);
\draw (v1) edge (v3);
\draw (v1) edge (v5);
\draw (v2) edge (v4);
\draw (v2) edge (v5);
\end{tikzpicture}
+ 
\begin{tikzpicture}[baseline={([yshift=-0.5ex]current bounding box.center)},scale=0.3,transform shape]
\node (v0) [circle,draw,minimum size = \minRadius] at (0,1.25) {};
\node (v1) [circle,draw,fill] at (0,0) {};
\node (v2) [circle,draw] at (-1,-1) {\huge $3$};
\node (v3) [circle,draw,fill] at (1,-1) {};
\node (v4) [circle,draw] at (0,-2) {\huge $1$};
\node (v5) [circle,draw] at (2,-2) {\huge $2$};
\draw (v0) edge (v1);
\draw (v1) edge (v2);
\draw (v1) edge (v3);
\draw (v3) edge (v4);
\draw (v3) edge (v5);
\end{tikzpicture}
+
\begin{tikzpicture}[baseline={([yshift=-0.5ex]current bounding box.center)},scale=0.3,transform shape]
\node (v0) [circle,draw,minimum size = \minRadius] at (0,1.25) {};
\node (v1) [circle,draw,fill] at (-1,0) {};
\node (v2) [circle,draw,fill] at (1,0) {};
\node (v3) [circle,draw] at (-2,-1) {\huge $3$};
\node (v4) [circle,draw] at (0,-1) {\huge $1$};
\node (v5) [circle,draw] at (2,-1) {\huge $2$};
\draw (v0) edge (v1);
\draw (v0) edge (v2);
\draw (v1) edge (v3);
\draw (v1) edge (v4);
\draw (v2) edge (v4);
\draw (v2) edge (v5);
\end{tikzpicture}
+
\begin{tikzpicture}[baseline={([yshift=-0.5ex]current bounding box.center)},scale=0.3,transform shape]
\node (v0) [circle,draw,minimum size = \minRadius] at (0,1.25) {};
\node (v1) [circle,draw,fill] at (-1,0) {};
\node (v2) [circle,draw,fill] at (1,0) {};
\node (v3) [circle,draw] at (-2,-1) {\huge $3$};
\node (v4) [circle,draw] at (0,-1) {\huge $1$};
\node (v5) [circle,draw] at (2,-1) {\huge $2$};
\draw (v0) edge (v1);
\draw (v0) edge (v2);
\draw (v1) edge (v3);
\draw (v1) edge (v5);
\draw (v2) edge (v4);
\draw (v2) edge (v5);
\end{tikzpicture}
+ 
\begin{tikzpicture}[baseline={([yshift=-0.5ex]current bounding box.center)},scale=0.3,transform shape]
\node (v0) [circle,draw,minimum size = \minRadius] at (0,1.25) {};
\node (v1) [circle,draw,fill] at (0,0) {};
\node (v2) [circle,draw] at (-1,-1) {\huge $2$};
\node (v3) [circle,draw,fill] at (1,-1) {};
\node (v4) [circle,draw] at (0,-2) {\huge $3$};
\node (v5) [circle,draw] at (2,-2) {\huge $1$};
\draw (v0) edge (v1);
\draw (v1) edge (v2);
\draw (v1) edge (v3);
\draw (v3) edge (v4);
\draw (v3) edge (v5);
\end{tikzpicture}
+
\begin{tikzpicture}[baseline={([yshift=-0.5ex]current bounding box.center)},scale=0.3,transform shape]
\node (v0) [circle,draw,minimum size = \minRadius] at (0,1.25) {};
\node (v1) [circle,draw,fill] at (-1,0) {};
\node (v2) [circle,draw,fill] at (1,0) {};
\node (v3) [circle,draw] at (-2,-1) {\huge $2$};
\node (v4) [circle,draw] at (0,-1) {\huge $3$};
\node (v5) [circle,draw] at (2,-1) {\huge $1$};
\draw (v0) edge (v1);
\draw (v0) edge (v2);
\draw (v1) edge (v3);
\draw (v1) edge (v4);
\draw (v2) edge (v4);
\draw (v2) edge (v5);
\end{tikzpicture}
+
\begin{tikzpicture}[baseline={([yshift=-0.5ex]current bounding box.center)},scale=0.3,transform shape]
\node (v0) [circle,draw,minimum size = \minRadius] at (0+10,1.25) {};
\node (v1) [circle,draw,fill] at (-1+10,0) {};
\node (v2) [circle,draw,fill] at (1+10,0) {};
\node (v3) [circle,draw] at (-2+10,-1) {\huge $2$};
\node (v4) [circle,draw] at (0+10,-1) {\huge $3$};
\node (v5) [circle,draw] at (2+10,-1) {\huge $1$};
\draw (v0) edge (v1);
\draw (v0) edge (v2);
\draw (v1) edge (v3);
\draw (v1) edge (v5);
\draw (v2) edge (v4);
\draw (v2) edge (v5);
\end{tikzpicture}
\ea


vanishes. This expression reduces to

\[
\begin{tikzpicture}[baseline={([yshift=-0.5ex]current bounding box.center)},scale=\scaleValue,transform shape]
\node (v0) [circle,draw,minimum size = \minRadius] at (0,1.25) {};
\node (v1) [circle,draw,fill] at (-1,0) {};
\node (v2) [circle,draw,fill] at (1,0) {};
\node (v3) [circle,draw] at (-2,-1) {\huge $1$};
\node (v4) [circle,draw] at (0,-1) {\huge $2$};
\node (v5) [circle,draw] at (2,-1) {\huge $3$};
\draw (v0) edge (v1);
\draw (v0) edge (v2);
\draw (v1) edge (v3);
\draw (v1) edge (v4);
\draw (v2) edge (v4);
\draw (v2) edge (v5);
\end{tikzpicture}
+
\begin{tikzpicture}[baseline={([yshift=-0.5ex]current bounding box.center)},scale=\scaleValue,transform shape]
\node (v0) [circle,draw,minimum size = \minRadius] at (0,1.25) {};
\node (v1) [circle,draw,fill] at (-1,0) {};
\node (v2) [circle,draw,fill] at (1,0) {};
\node (v3) [circle,draw] at (-2,-1) {\huge $1$};
\node (v4) [circle,draw] at (0,-1) {\huge $2$};
\node (v5) [circle,draw] at (2,-1) {\huge $3$};
\draw (v0) edge (v1);
\draw (v0) edge (v2);
\draw (v1) edge (v3);
\draw (v1) edge (v5);
\draw (v2) edge (v4);
\draw (v2) edge (v5);
\end{tikzpicture}
+
\begin{tikzpicture}[baseline={([yshift=-0.5ex]current bounding box.center)},scale=\scaleValue,transform shape]
\node (v0) [circle,draw,minimum size = \minRadius] at (0,1.25) {};
\node (v1) [circle,draw,fill] at (-1,0) {};
\node (v2) [circle,draw,fill] at (1,0) {};
\node (v3) [circle,draw] at (-2,-1) {\huge $3$};
\node (v4) [circle,draw] at (0,-1) {\huge $1$};
\node (v5) [circle,draw] at (2,-1) {\huge $2$};
\draw (v0) edge (v1);
\draw (v0) edge (v2);
\draw (v1) edge (v3);
\draw (v1) edge (v4);
\draw (v2) edge (v4);
\draw (v2) edge (v5);
\end{tikzpicture}
+
\begin{tikzpicture}[baseline={([yshift=-0.5ex]current bounding box.center)},scale=\scaleValue,transform shape]
\node (v0) [circle,draw,minimum size = \minRadius] at (0,1.25) {};
\node (v1) [circle,draw,fill] at (-1,0) {};
\node (v2) [circle,draw,fill] at (1,0) {};
\node (v3) [circle,draw] at (-2,-1) {\huge $3$};
\node (v4) [circle,draw] at (0,-1) {\huge $1$};
\node (v5) [circle,draw] at (2,-1) {\huge $2$};
\draw (v0) edge (v1);
\draw (v0) edge (v2);
\draw (v1) edge (v3);
\draw (v1) edge (v5);
\draw (v2) edge (v4);
\draw (v2) edge (v5);
\end{tikzpicture}
+
\begin{tikzpicture}[baseline={([yshift=-0.5ex]current bounding box.center)},scale=\scaleValue,transform shape]
\node (v0) [circle,draw,minimum size = \minRadius] at (0,1.25) {};
\node (v1) [circle,draw,fill] at (-1,0) {};
\node (v2) [circle,draw,fill] at (1,0) {};
\node (v3) [circle,draw] at (-2,-1) {\huge $2$};
\node (v4) [circle,draw] at (0,-1) {\huge $3$};
\node (v5) [circle,draw] at (2,-1) {\huge $1$};
\draw (v0) edge (v1);
\draw (v0) edge (v2);
\draw (v1) edge (v3);
\draw (v1) edge (v4);
\draw (v2) edge (v4);
\draw (v2) edge (v5);
\end{tikzpicture}
+
\begin{tikzpicture}[baseline={([yshift=-0.5ex]current bounding box.center)},scale=\scaleValue,transform shape]
\node (v0) [circle,draw,minimum size = \minRadius] at (0+10,1.25) {};
\node (v1) [circle,draw,fill] at (-1+10,0) {};
\node (v2) [circle,draw,fill] at (1+10,0) {};
\node (v3) [circle,draw] at (-2+10,-1) {\huge $2$};
\node (v4) [circle,draw] at (0+10,-1) {\huge $3$};
\node (v5) [circle,draw] at (2+10,-1) {\huge $1$};
\draw (v0) edge (v1);
\draw (v0) edge (v2);
\draw (v1) edge (v3);
\draw (v1) edge (v5);
\draw (v2) edge (v4);
\draw (v2) edge (v5);
\end{tikzpicture}.
\]


If $d$ is odd we observe that


\[
\begin{tikzpicture}[baseline={([yshift=-0.5ex]current bounding box.center)},scale=\scaleValue,transform shape]
\node (v0) [circle,draw,minimum size = \minRadius] at (0+10,1.25) {};
\node (v1) [circle,draw,fill] at (-1+10,0) {};
\node (v2) [circle,draw,fill] at (1+10,0) {};
\node (v3) [circle,draw] at (-2+10,-1) {\huge $1$};
\node (v4) [circle,draw] at (0+10,-1) {\huge $2$};
\node (v5) [circle,draw] at (2+10,-1) {\huge $3$};
\draw (v0) edge (v1);
\draw (v0) edge (v2);
\draw (v1) edge (v3);
\draw (v1) edge (v5);
\draw (v2) edge (v4);
\draw (v2) edge (v5);
\end{tikzpicture}
=
-\begin{tikzpicture}[baseline={([yshift=-0.5ex]current bounding box.center)},scale=\scaleValue,transform shape]
\node (v0) [circle,draw,minimum size = \minRadius] at (0,1.25) {};
\node (v1) [circle,draw,fill] at (-1,0) {};
\node (v2) [circle,draw,fill] at (1,0) {};
\node (v3) [circle,draw] at (-2,-1) {\huge $2$};
\node (v4) [circle,draw] at (0,-1) {\huge $3$};
\node (v5) [circle,draw] at (2,-1) {\huge $1$};
\draw (v0) edge (v1);
\draw (v0) edge (v2);
\draw (v1) edge (v3);
\draw (v1) edge (v4);
\draw (v2) edge (v4);
\draw (v2) edge (v5);
\end{tikzpicture}
\]


\noi and thus for any permutation of $\{1,2,3\}$. For $d$ even we see that


\[
\begin{tikzpicture}[baseline={([yshift=-0.5ex]current bounding box.center)},scale=\scaleValue,transform shape]
\node (v0) [circle,draw,minimum size = \minRadius] at (0+10,1.25) {};
\node (v1) [circle,draw,fill] at (-1+10,0) {};
\node (v2) [circle,draw,fill] at (1+10,0) {};
\node (v3) [circle,draw] at (-2+10,-1) {\huge $1$};
\node (v4) [circle,draw] at (0+10,-1) {\huge $2$};
\node (v5) [circle,draw] at (2+10,-1) {\huge $3$};
\draw (v0) edge (v1);
\draw (v0) edge (v2);
\draw (v1) edge (v3);
\draw (v1) edge (v5);
\draw (v2) edge (v4);
\draw (v2) edge (v5);
\node at (8.5,0) {\Large $1$};
\node at (11.5,0) {\Large $2$};
\end{tikzpicture}
=-\begin{tikzpicture}[baseline={([yshift=-0.5ex]current bounding box.center)},scale=\scaleValue,transform shape]
\node (v0) [circle,draw,minimum size = \minRadius] at (0,1.25) {};
\node (v1) [circle,draw,fill] at (-1,0) {};
\node (v2) [circle,draw,fill] at (1,0) {};
\node (v3) [circle,draw] at (-2,-1) {\huge $2$};
\node (v4) [circle,draw] at (0,-1) {\huge $3$};
\node (v5) [circle,draw] at (2,-1) {\huge $1$};
\draw (v0) edge (v1);
\draw (v0) edge (v2);
\draw (v1) edge (v3);
\draw (v1) edge (v4);
\draw (v2) edge (v4);
\draw (v2) edge (v5);
\node at (-1.5,0) {\Large $1$};
\node at (1.5,0) {\Large $2$};
\end{tikzpicture}
\]


In any case we see that the Jacobi identity is mapped to zero.

\end{proof}

\begin{rem}

The map 


\ba{rcc}
\assOp & \longrightarrow & \oFun(\assOp) \\[5pt]
\begin{tikzpicture}[baseline={([yshift=-0.5ex]current bounding box.center)},scale=\scaleValue,transform shape]
\node (v0) at (0,1) {};
\node (v1) [circle,draw] at (0,0) {};
\node (v2) at (-1,-1) {\huge $1$};
\node (v3) at (1,-1) {\huge $2$};
\draw (v0) edge (v1);
\draw (v1) edge (v2);
\draw (v1) edge (v3);
\end{tikzpicture} & \longmapsto & \begin{tikzpicture}[baseline={([yshift=-0.5ex]current bounding box.center)},scale=\scaleValue,transform shape]
\node (v0) [circle,draw,minimum size = \minRadius] at (0,1) {};
\node (v1) [circle,draw] at (0,0) {};
\node (v2) [circle,draw] at (-1,-1) {\huge $1$};
\node (v3) [circle,draw] at (1,-1) {\huge $2$};
\draw (v0) edge (v1);
\draw (v1) edge (v2);
\draw (v1) edge (v3);
\end{tikzpicture}
\ea


\noi is not a morphism of operads as the associativity condition is not mapped to zero. Indeed the associativity condition would require that the following relation vanishes,


\ba{rl}
\begin{tikzpicture}[baseline={([yshift=-0.66ex]current bounding box.center)},scale=\scaleValue,transform shape]
\node (v0) [circle,draw,minimum size = \minRadius] at (0,1.25) {};
\node (v1) [circle,draw] at (0,0) {};
\node (v2) [circle,draw] at (-1,-1) {};
\node (v3) [circle,draw] at (1,-1) {\huge $3$};
\node (v4) [circle,draw] at (-2,-2) {\huge $1$};
\node (v5) [circle,draw] at (0,-2) {\huge $2$};
\draw (v0) edge (v1);
\draw (v1) edge (v2);
\draw (v1) edge (v3);
\draw (v2) edge (v4);
\draw (v2) edge (v5);
\end{tikzpicture}
+
\begin{tikzpicture}[baseline={([yshift=-0.5ex]current bounding box.center)},scale=\scaleValue,transform shape]
\node (v0) [circle,draw,minimum size = \minRadius] at (0,1.25) {};
\node (v1) [circle,draw] at (-1,0) {};
\node (v2) [circle,draw] at (1,0) {};
\node (v3) [circle,draw] at (-2,-1) {\huge $1$};
\node (v4) [circle,draw] at (0,-1) {\huge $2$};
\node (v5) [circle,draw] at (2,-1) {\huge $3$};
\draw (v0) edge (v1);
\draw (v0) edge (v2);
\draw (v1) edge (v3);
\draw (v1) edge (v4);
\draw (v2) edge (v4);
\draw (v2) edge (v5);
\end{tikzpicture}
+
\begin{tikzpicture}[baseline={([yshift=-0.5ex]current bounding box.center)},scale=\scaleValue,transform shape]
\node (v0) [circle,draw,minimum size = \minRadius] at (0,1.25) {};
\node (v1) [circle,draw] at (-1,0) {};
\node (v2) [circle,draw] at (1,0) {};
\node (v3) [circle,draw] at (-2,-1) {\huge $1$};
\node (v4) [circle,draw] at (0,-1) {\huge $2$};
\node (v5) [circle,draw] at (2,-1) {\huge $3$};
\draw (v0) edge (v1);
\draw (v0) edge (v2);
\draw (v1) edge (v3);
\draw (v1) edge (v4);
\draw (v2) edge (v3);
\draw (v2) edge (v5);
\end{tikzpicture}
-
\begin{tikzpicture}[baseline={([yshift=-0.5ex]current bounding box.center)},scale=\scaleValue,transform shape]
\node (v0) [circle,draw,minimum size = \minRadius] at (0,1.25) {};
\node (v1) [circle,draw] at (0,0) {};
\node (v2) [circle,draw] at (-1,-1) {\huge $1$};
\node (v3) [circle,draw] at (1,-1) {};
\node (v4) [circle,draw] at (0,-2) {\huge $2$};
\node (v5) [circle,draw] at (2,-2) {\huge $3$};
\draw (v0) edge (v1);
\draw (v1) edge (v2);
\draw (v1) edge (v3);
\draw (v3) edge (v4);
\draw (v3) edge (v5);
\end{tikzpicture}
-
\begin{tikzpicture}[baseline={([yshift=-0.5ex]current bounding box.center)},scale=\scaleValue,transform shape]
\node (v0) [circle,draw,minimum size = \minRadius] at (0,1.25) {};
\node (v1) [circle,draw] at (-1,0) {};
\node (v2) [circle,draw] at (1,0) {};
\node (v3) [circle,draw] at (-2,-1) {\huge $1$};
\node (v4) [circle,draw] at (0,-1) {\huge $2$};
\node (v5) [circle,draw] at (2,-1) {\huge $3$};
\draw (v0) edge (v1);
\draw (v0) edge (v2);
\draw (v1) edge (v3);
\draw (v1) edge (v4);
\draw (v2) edge (v4);
\draw (v2) edge (v5);
\end{tikzpicture}
-
\begin{tikzpicture}[baseline={([yshift=-0.5ex]current bounding box.center)},scale=\scaleValue,transform shape]
\node (v0) [circle,draw,minimum size = \minRadius] at (0,1.25) {};
\node (v1) [circle,draw] at (-1,0) {};
\node (v2) [circle,draw] at (1,0) {};
\node (v3) [circle,draw] at (-2,-1) {\huge $1$};
\node (v4) [circle,draw] at (0,-1) {\huge $2$};
\node (v5) [circle,draw] at (2,-1) {\huge $3$};
\draw (v0) edge (v1);
\draw (v0) edge (v2);
\draw (v1) edge (v3);
\draw (v1) edge (v5);
\draw (v2) edge (v4);
\draw (v2) edge (v5);
\end{tikzpicture} & \\[20pt]
=&
\begin{tikzpicture}[baseline={([yshift=-0.5ex]current bounding box.center)},scale=\scaleValue,transform shape]
\node (v0) [circle,draw,minimum size = \minRadius] at (0,1.25) {};
\node (v1) [circle,draw] at (-1,0) {};
\node (v2) [circle,draw] at (1,0) {};
\node (v3) [circle,draw] at (-2,-1) {\huge $1$};
\node (v4) [circle,draw] at (0,-1) {\huge $2$};
\node (v5) [circle,draw] at (2,-1) {\huge $3$};
\draw (v0) edge (v1);
\draw (v0) edge (v2);
\draw (v1) edge (v3);
\draw (v1) edge (v4);
\draw (v2) edge (v3);
\draw (v2) edge (v5);
\end{tikzpicture}
 -
\begin{tikzpicture}[baseline={([yshift=-0.5ex]current bounding box.center)},scale=\scaleValue,transform shape]
\node (v0) [circle,draw,minimum size = \minRadius] at (0,1.25) {};
\node (v1) [circle,draw] at (-1,0) {};
\node (v2) [circle,draw] at (1,0) {};
\node (v3) [circle,draw] at (-2,-1) {\huge $1$};
\node (v4) [circle,draw] at (0,-1) {\huge $2$};
\node (v5) [circle,draw] at (2,-1) {\huge $3$};
\draw (v0) edge (v1);
\draw (v0) edge (v2);
\draw (v1) edge (v3);
\draw (v1) edge (v5);
\draw (v2) edge (v4);
\draw (v2) edge (v5);
\end{tikzpicture}
\ea


\noi which is not the case.


\end{rem}

\end{section}

\begin{section}{Deformation complexes}

We follow \cite{MR2572248} to introduce deformation complexes. We consider a  morphism of operads $f:\LieOp_d \rightarrow \mathcal{P}$ and define the associated deformation complex by


\ba{lllll}
\deform(\LieOp_d \stackrel{f}{\ra} \mcl{P}) & = & \deform(\hoLie_d \stackrel{\bar{f}}{\ra} \mcl{P}) & = & \prod_{n \geq 1}{[E(n)^\ast \otimes \mcl{P}(n)]^{\mbb{S}_n}[-1]} \\
& & & = & \prod_{n \geq 1}{[\text{sgn}_n^{\otimes |d|} \otimes \mcl{P}(n)]^{\mbb{S}_n}[d-dn]}.
\ea


The map $\bar{f}:\hoLie_d \rightarrow \LieOp_d \stackrel{f}{\ra} \mcl{P}$ induces a differential as follows: an element $F \in [\text{sgn}_n^{\otimes |d|} \otimes \mcl{P}(n)]^{\mbb{S}_n}$ can be interpreted as the image of 

\[
\begin{tikzpicture}[baseline={([yshift=-0.5ex]current bounding box.center)},scale=\scaleValue,transform shape]
\node  [circle,draw,fill] (v1) at (0,0) {};
\node (v2) at (0,1) {};
\node (v4) at (-1,-1) {\huge $2$};
\node (v3) at (-2,-1) {\huge $1$};
\node (v5) at (1,-1) {};
\node (v6) at (2,-1) {\huge $n$};
\node at (0,-1) {\huge $\cdots$};
\draw  (v1) edge (v2);
\draw  (v1) edge (v3);
\draw  (v1) edge (v4);
\draw  (v1) edge (v5);
\draw  (v1) edge (v6);
\end{tikzpicture}
\]


\noi under a derivation $F:\hoLie_d \ra \mcl{P}$. The differential is given by


\[
\delta F = \sum_{n=n'+n''}{\includegraphics[scale=0.5,valign=c]{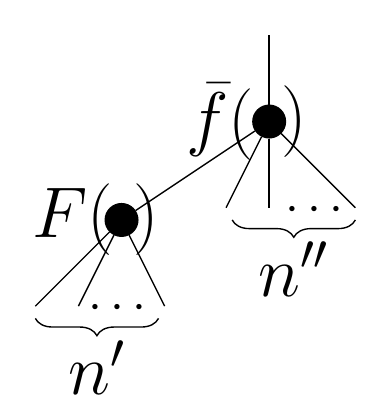}-(-1)^{|F|}\includegraphics[scale=0.5,valign=c]{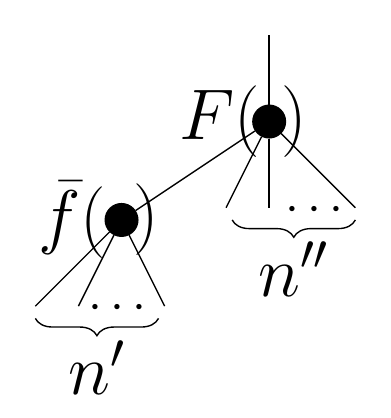}}.
\]


We recall the following facts:


\begin{fact}\cite{MR3348138}

\[
H(\mathbf{Def}(\LieOp_d \rightarrow \LieOp_d))=\mathbb{K}\langle\begin{tikzpicture}[baseline={([yshift=-0.5ex]current bounding box.center)},scale=\scaleValue,transform shape]
\node (v0) at (0,1) {};
\node (v1) [circle,draw,fill] at (0,0) {};
\node (v2) at (-1,-1) {\huge $1$};
\node (v3) at (1,-1) {\huge $2$};
\draw (v0) edge (v1);
\draw (v1) edge (v2);
\draw (v1) edge (v3);
\end{tikzpicture}\rangle.
\]

\end{fact}

\begin{fact}\cite{MR2572248}

Let $f:\mcl{P} \ra \mcl{Q}_1$, $s:\mcl{Q}_1 \ra \mcl{Q}_2$ be morphism of dg operads such that $s$ is a quasi-isomorphism. Then there is a quasi-isomorphism


\[
\deform(\mcl{P} \stackrel{f}{\ra} \mcl{Q}_1) \cong \deform(\mcl{P} \stackrel{sf}{\ra} \mcl{Q}_2).
\]

\end{fact}

Combining these two facts we see that $H(\deform(\hoLie_d \ra \hoLie_d))$ is one-dimensional.

\sk

Now consider for $\lambda \in \mbb{K}^\ast$ the map $\phi_\lambda$, defined by


\[
\phi_\lambda(\begin{tikzpicture}[baseline={([yshift=-0.5ex]current bounding box.center)},scale=\scaleValue,transform shape]
\node [circle,draw,fill] (v1) at (0,0) {};
\node (v2) at (-1,-1.5) {\huge $1$};
\node (v3) at (1,-1.5) {\huge $n$};
\node at (0.055,-1) {\huge $\cdots$};
\draw  (v1) edge (v2);
\draw  (v1) edge (v3);
\node (v4) at (0,1.25) {};
\draw  (v1) edge (v4);
\end{tikzpicture}
):=\lambda^{n-1}\begin{tikzpicture}[baseline={([yshift=-0.5ex]current bounding box.center)},scale=\scaleValue,transform shape]
\node [circle,draw,fill] (v1) at (0,0) {};
\node (v2) at (-1,-1.5) {\huge $1$};
\node (v3) at (1,-1.5) {\huge $n$};
\node at (0.055,-1) {\huge $\cdots$};
\draw  (v1) edge (v2);
\draw  (v1) edge (v3);
\node (v4) at (0,1.25) {};
\draw  (v1) edge (v4);
\end{tikzpicture}
.
\]


This is an automorphism of $\hoLie_d$. Compute


\[
\frac{d\phi_\lambda}{d \lambda}\Bigr|_{\lambda=1}(\begin{tikzpicture}[baseline={([yshift=-0.5ex]current bounding box.center)},scale=\scaleValue,transform shape]
\node [circle,draw,fill] (v1) at (0,0) {};
\node (v2) at (-1,-1.5) {\huge $1$};
\node (v3) at (1,-1.5) {\huge $n$};
\node at (0.055,-1) {\huge $\cdots$};
\draw  (v1) edge (v2);
\draw  (v1) edge (v3);
\node (v4) at (0,1.25) {};
\draw  (v1) edge (v4);
\end{tikzpicture})=(n-1)\begin{tikzpicture}[baseline={([yshift=-0.5ex]current bounding box.center)},scale=\scaleValue,transform shape]
\node [circle,draw,fill] (v1) at (0,0) {};
\node (v2) at (-1,-1.5) {\huge $1$};
\node (v3) at (1,-1.5) {\huge $n$};
\node at (0.055,-1) {\huge $\cdots$};
\draw  (v1) edge (v2);
\draw  (v1) edge (v3);
\node (v4) at (0,1.25) {};
\draw  (v1) edge (v4);
\end{tikzpicture}.
\]


This implies that $\Lambda:=\sum_{n \geq 2}{(n-1)\begin{tikzpicture}[baseline={([yshift=-0.5ex]current bounding box.center)},scale=\scaleValue,transform shape]
\node [circle,draw,fill] (v1) at (0,0) {};
\node (v2) at (-1,-1.5) {\huge $1$};
\node (v3) at (1,-1.5) {\huge $n$};
\node at (0.055,-1) {\huge $\cdots$};
\draw  (v1) edge (v2);
\draw  (v1) edge (v3);
\node (v4) at (0,1.25) {};
\draw  (v1) edge (v4);
\end{tikzpicture}} \in \deform(\hoLie_d \ra \hoLie_d)$ is a cohomology class and


\[
H^0(\deform(\hoLie_d \ra \hoLie_d))=\mbb{K}\langle \Lambda \rangle.
\]

\end{section}

We are interested in the deformation complex $\deform(\LieOp_d \stackrel{i}{\ra} \oLie{d})$ of the morphism of operads described in $(\ref{morphism_lie_olie})$. The elements of this complex can be described by formal series of elements in $\oLie{d}$ with either unlabelled white input vertices if $d$ is even or with an ordering of the input vertices up to even permutations. The differential can be described more explicitly for $\Gamma \in \deform(\LieOp_d \stackrel{i}{\ra} \oLie{d})$ by

\[
\delta \Gamma=\begin{tikzpicture}[baseline={([yshift=-0.5ex]current bounding box.center)},scale=\scaleValue,transform shape]
\node [circle,draw,fill] (v1) at (0,0) {};
\node [circle,draw,minimum size = \minRadius] (v2) at (-1,-1) {};
\node [circle,draw,minimum size = \minRadius] (v3) at (1,-1) {};
\node [circle,draw,minimum size = \minRadius] (v4) at (0,1.25) {};
\draw  (v1) edge (v2);
\draw  (v1) edge (v3);
\draw  (v1) edge (v4);
\end{tikzpicture} \, o \, \Gamma
- (-1)^{|\Gamma|} \sum_{v \in V_o(\Gamma)} \Gamma \, o_v \, \begin{tikzpicture}[baseline={([yshift=-0.5ex]current bounding box.center)},scale=\scaleValue,transform shape]
\node [circle,draw,fill] (v1) at (0,0) {};
\node [circle,draw,minimum size = \minRadius] (v2) at (-1,-1) {};
\node [circle,draw,minimum size = \minRadius] (v3) at (1,-1) {};
\node [circle,draw,minimum size = \minRadius] (v4) at (0,1.25) {};
\draw  (v1) edge (v2);
\draw  (v1) edge (v3);
\draw  (v1) edge (v4);
\end{tikzpicture},
\]

where $V_o(\Gamma)$ denotes the set of white input vertices of $\Gamma$.


\begin{section}{Kontsevich graph complexes} 

We recall some facts about the operad of graphs and graph complexes. For completeness we recall the definition of a directed graph:


\begin{defi}

A graph with hairs is a triple $\Gamma=(H(\Gamma,\sqcup,\tau)$ where

\begin{enumerate}[label=\arabic*)]

\item $H(\Gamma)$ is a finite set of \textit{halfedges},

\item $\sqcup$ is a partition of $H(\Gamma)$


\[
H(\Gamma)=\bigsqcup_{v \in V(\Gamma)}{H(v)},
\]


\noi parametrized by a set $V(\Gamma)$ called the \textit{set of vertices} of $\Gamma$. For a vertex $v$ the set $H(v)$ is called the \textit{set of halfedges attached to $v$}. The valency of a vertex $v$ is defined to be the cardinality of $H(v)$.

\item $\tau:H(\Gamma)  \ra H(\Gamma)$ is an involution. The orbits of cardinality $2$ are called the \textit{(internal) edges} and the set of edges is denoted by $E(\Gamma)$. The orbits of cardinality $1$ are called \textit{hairs} and the set of hairs is denoted by $L(\Gamma)$

\end{enumerate}

If $L(\Gamma)=\emptyset$, then $\Gamma$ is simply called a \textit{graph}. A graph $\Gamma$ is called \textit{directed} if each edge $e=(h,\tau(h))$ comes with a choice of an ordering of halfedges.

\end{defi}


We denote by $G_{v,e}$ the set of connected directed graphs with no hairs, with $v$ vertices, labelled by $\{1,\cdots,v\}$, and $e$ edges labelled by $\{1,\cdots,e\}$. We define a $\mbb{S}_v$ module


\begin{equation*}
\gra_d(v):=\left\{ \begin{array}{ll}
\bigoplus_{e \geq 0}{\mbb{K}\langle G_{v,e} \rangle \otimes_{\mbb{S}_e \ltimes (\mbb{S}_2)^e} \text{sgn}_e [e(d-1)]} & \text{if $d$ is even,} \\
\bigoplus_{e \geq 0}{\mbb{K}\langle G_{v,e} \rangle \otimes_{\mbb{S}_e \ltimes (\mbb{S}_2)^e} \text{sgn}^{\otimes e}_2 [e(d-1)]} & \text{if $d$ is odd}
\end{array} \right.
\end{equation*}


\noi where $d$ is an integer, and thus an operad $\gra_d=\{\gra_d(v)\}_{v \geq 1}$ called the operad of graphs.


If $d$ is even, elements of $\gra_d$ can be seen as undirected graphs with edges having degree $1-d$ together with an ordering of the edges up to an even permutation, while an odd permutation acts by multiplication by $-1$.


If $d$ is odd, $\gra_d$ consists of directed graphs where changing the direction of an edge yields a multiplication by $-1$.


The operadic composition $\Gamma_1 \, \text{o}_v \, \Gamma_2$ works by substituting the graph $\Gamma_2$ in the vertex $v$ of $\Gamma_1$ and by summing over all possibilities of attaching the edges of $v$ to the vertices of $\Gamma_2$.

\begin{ex}
\[
\begin{tikzpicture}[baseline={([yshift=-0.5ex]current bounding box.center)},scale=\scaleValue,transform shape]
\node (v0) [circle,draw,fill,label=left: \huge $1$] at (0,0)  {};
\node (v1) [circle,draw,fill,label=right: \huge $2$] at (2,0) {};
\node (v2) [circle,draw,fill,label=left: \huge $3$] at (0,-2) {};
\draw (v0) edge (v1);
\draw (v0) edge (v2);
\end{tikzpicture}
\, o_1 \,
\begin{tikzpicture}[baseline={([yshift=-0.5ex]current bounding box.center)},scale=\scaleValue,transform shape]
\node (v0) [circle,draw,fill,label=left: \huge $1$] at (0,0) {};
\node (v1) [circle,draw,fill,label=right: \huge $2$] at (2,0) {};
\draw (v0) edge (v1);
\end{tikzpicture}
=
\begin{tikzpicture}[baseline={([yshift=-0.5ex]current bounding box.center)},scale=\scaleValue,transform shape]
\node (v1) [circle,draw,fill,label=left: \huge $1$] at (0,0) {};
\node (v2) [circle,draw,fill,label=right: \huge $2$] at (2,0) {};
\node (v3) [circle,draw,fill,label=left: \huge $3$] at (1,1) {};
\node (v4) [circle,draw,fill,label=left: \huge $4$] at (1,-1) {};
\draw (v2) edge (v1);
\draw (v3) edge (v1);
\draw (v4) edge (v1);
\end{tikzpicture}
+
\begin{tikzpicture}[baseline={([yshift=-0.5ex]current bounding box.center)},scale=\scaleValue,transform shape]
\node (v1) [circle,draw,fill,label=left: \huge $1$] at (0,0) {};
\node (v2) [circle,draw,fill,label=right: \huge $2$] at (2,0) {};
\node (v3) [circle,draw,fill,label=left: \huge $3$] at (1,1) {};
\node (v4) [circle,draw,fill,label=left: \huge $4$] at (1,-1) {};
\draw (v2) edge (v1);
\draw (v3) edge (v2);
\draw (v4) edge (v1);
\end{tikzpicture}
+
\begin{tikzpicture}[baseline={([yshift=-0.5ex]current bounding box.center)},scale=\scaleValue,transform shape]
\node (v1) [circle,draw,fill,label=left: \huge $1$] at (0,0) {};
\node (v2) [circle,draw,fill,label=right: \huge $2$] at (2,0) {};
\node (v3) [circle,draw,fill,label=left: \huge $3$] at (1,1) {};
\node (v4) [circle,draw,fill,label=left: \huge $4$] at (1,-1) {};
\draw (v2) edge (v1);
\draw (v3) edge (v1);
\draw (v4) edge (v2);
\end{tikzpicture}
+
\begin{tikzpicture}[baseline={([yshift=-0.5ex]current bounding box.center)},scale=\scaleValue,transform shape]
\node (v1) [circle,draw,fill,label=left: \huge $1$] at (0,0) {};
\node (v2) [circle,draw,fill,label=right: \huge $2$] at (2,0) {};
\node (v3) [circle,draw,fill,label=left: \huge $3$] at (1,1) {};
\node (v4) [circle,draw,fill,label=left: \huge $4$] at (1,-1) {};
\draw (v2) edge (v1);
\draw (v3) edge (v2);
\draw (v4) edge (v2);
\end{tikzpicture}
\]
\end{ex}


There is a morphism of operads \cite{MR3348138}


\[
\LieOp_d \lra \gra_d
\]


\noi by mapping


\begin{equation*}
\begin{tikzpicture}[baseline={([yshift=-0.5ex]current bounding box.center)},scale=\scaleValue,transform shape]
\node (v0) at (0,1) {};
\node (v1) [circle,draw,fill] at (0,0) {};
\node (v2) at (-1,-1) {\huge $1$};
\node (v3) at (1,-1) {\huge $2$};
\draw (v1) edge (v0);
\draw (v1) edge (v2);
\draw (v1) edge (v3);
\end{tikzpicture}
\lma \left\{ \begin{array}{cl}
\begin{tikzpicture}[baseline={([yshift=-0.5ex]current bounding box.center)},scale=\scaleValue,transform shape]
\node (v0) [draw,fill,circle,label=above:\huge $1$] at (0,0) {};
\node (v1) [draw,fill,circle,label=above:\huge $2$] at (2,0) {};
\draw [->] (v0) edge (v1);
\end{tikzpicture} & \text{if $d$ is odd} \\
\begin{tikzpicture}[baseline={([yshift=-0.5ex]current bounding box.center)},scale=\scaleValue,transform shape]
\node (v0) [draw,fill,circle,label=above:\huge $1$] at (0,0) {};
\node (v1) [draw,fill,circle,label=above:\huge $2$] at (2,0) {};
\draw (v0) edge (v1);
\end{tikzpicture} & \text{if $d$ is even.}
\end{array} \right. 
\end{equation*}


Thus we can define the \textit{full graph complex} as the deformation complex


\[
\fgc_d:=\deform(\LieOp_d \lra \gra_d).
\]


Elements of $\fgc_d$ can be interpreted as connected graphs together with an ordering of the edges, up to an even permutation, when $d$ is even. For $d$ odd, we have an ordering of the vertices, up to even permutations, together with a choice on the orientation of each edge, up to a flip yielding a multiplication by $-1$.


Its differential can be represented by

\[
\delta(\Gamma)=-2\sum_{v \in V(\Gamma)}{\begin{tikzpicture}[baseline={([yshift=-0.5ex]current bounding box.center)},scale=\scaleValue,transform shape]
\node (v1) [circle,draw,fill] at (0,0.5) {};
\node (v2) [circle,draw] at (0,-1) {\huge $\Gamma$};
\draw (v1) edge (v2);
\end{tikzpicture}}+\sum_{v \in V_o(\Gamma)}{\Gamma \, \text{o}_v \, \begin{tikzpicture}[baseline={([yshift=-0.5ex]current bounding box.center)},scale=\scaleValue,transform shape]
\node (v1) [circle,draw,fill] at (0,0) {};
\node (v2) [circle,draw,fill] at (2,0) {};
\draw (v1) edge (v2);
\end{tikzpicture}},
\]


\noi where the first term is given by summing over all attachments of a univalent vertex to a vertex of $\Gamma$.


We consider the subcomplex $\gc_d \subset \fgc_d$ spanned by connected graphs with vertices having valency at least $3$. Let $(\gc_d^2:=\bigoplus_{\substack{p \geq 1 \\ p \equiv 2d+1 mod 4}}{\mbb{K}[d-p]},0)$. It has been shown \cite{MR3348138} that the natural projection $\fgc_d \ra \gc_d^{\geq 2}:=\gc_d \oplus \gc_d^2$is a quasi-isomorphism, i.e. $H(\gc^{\geq 2}_d)=H(\fgc_d)$. In addition we recall the following theorem concerning the cohomology of $\gc_d$.


\begin{thm}{\bf{T. Willwacher},\cite{MR3348138}}

For $d=2$ one has


\[
H^0(\gc_2)=\grt_1,
\]


\noi where $\grt_1$ denotes the Grothendieck-Teichmüller Lie algebra (see below).

\end{thm}


We recall the definition of the Grothendieck-Teichmüller group \cite{MR3348138}.


Consider $\mbb{F}_2=\mbb{K}\langle \langle X,Y \rangle \rangle$ the completed free associative algebra generated by $X,Y$. Define a coproduct $\Delta$ by setting $\Delta X=X \otimes 1 + 1 \otimes X$ and $\Delta Y = Y \otimes 1 + 1 \otimes Y$. An element $\alpha \in \mbb{F}_2$ is called \textit{group-like} if $\Delta \alpha =\alpha \otimes \alpha$.


We also need the Drinfeld-Kohno Lie algebra $t_n$ defined by generators $t_{ij}=t_{ji}$ for $1 \leq i j \leq n$ ,$i \neq j$, which satisfy $[t_{ij},t_{ik}+t_{kj}]=0$ and $[t_{ij},t_{kl}]=0$ for any distinct $i,j,k,l$.


The Grothendieck-Teichmüller group $\GRT$ is defined to be set of group-like elements $\alpha \in \mbb{F}_2$ which satisfy the following relations


\ba{rcl}
\alpha(t_{12},t_{23}+t_{24})\alpha(t_{13}+t_{23},t_{34}) & = & \alpha(t_{23},t_{34}) \alpha(t_{13}+t_{13},t_{24}+t_{34}) \alpha(t_{12},t_{23}) \\
1 & = & \alpha(t_{13},t_{12}) \alpha(t_{13},t_{23})^{-1} \alpha(t_{12},t_{23}) \\
\alpha(x,y) &=& \alpha(y,x)^{-1}.
\ea


The group structure is given by


\[
\alpha_1(X,Y) \cdot \alpha_2(Y,X) = \alpha_1(X,Y) \alpha_2(X,\alpha_1(X,Y)^{-1}Y\alpha_1(X,Y)).
\]


The Grothendieck-Teichmüller Lie algebra $\grt_1$ is given by the elements $\alpha \in \hat{\mbb{F}_2}(X,Y)$, the completed free Lie algebra generated by $X,Y$, satisfying

\ba{rcl}
\alpha(t_{12},t_{23}+t_{24}) + \alpha(t_{13}+t_{23},t_{34}) &=& \alpha(t_{23},t_{24}) + \alpha(t_{12}+t_{13},t_{24}+t_{34}) + \alpha(t_{12},t_{23}) \\
\alpha(X,Y) + \alpha(Y,-X-Y) + \alpha(-X-Y,X) &=&0 \\
\alpha(X,Y) &=& - \alpha(Y,X)
\ea


It has been shown \cite{Rossi2014} that $H(\gc_2)$ contains the so-called \textit{wheel classes} $\mfk{w}_{2n+1}$ given by


\ba{rcl}
\mfk{w}_3 &=& \begin{tikzpicture}[baseline={([yshift=-0.5ex]current bounding box.center)},scale=\scaleValue,transform shape]
\node [circle,draw,fill] (v1) at (0.5,0.5) {};
\node [circle,draw,fill] (v3) at (-1,-1.5) {};
\node [circle,draw,fill] (v4) at (2,-1.5) {};
\node [circle,draw,fill] (v2) at (0.5,-0.7) {};
\draw  (v1) edge (v2);
\draw  (v3) edge (v2);
\draw  (v3) edge (v1);
\draw  (v3) edge (v4);
\draw  (v4) edge (v1);
\draw  (v2) edge (v4);
\end{tikzpicture} \\
\mfk{w}_5 &=& \begin{tikzpicture}[baseline={([yshift=-0.5ex]current bounding box.center)},scale=\scaleValue,transform shape]
\node [circle,draw,fill] (v2) at (0.5,1) {};
\node [circle,draw,fill] (v1) at (-1,0) {};
\node [circle,draw,fill] (v3) at (2,0) {};
\node [circle,draw,fill] (v5) at (-0.5,-1.5) {};
\node [circle,draw,fill] (v4) at (1.5,-1.5) {};
\node [circle,draw,fill] (v6) at (0.5,-0.5) {};
\draw  (v1) edge (v2);
\draw  (v2) edge (v3);
\draw  (v3) edge (v4);
\draw  (v5) edge (v4);
\draw  (v5) edge (v1);
\draw  (v6) edge (v2);
\draw  (v6) edge (v3);
\draw  (v4) edge (v6);
\draw  (v6) edge (v5);
\draw  (v6) edge (v1);
\end{tikzpicture}+\frac{5}{2}\begin{tikzpicture}[baseline={([yshift=-0.5ex]current bounding box.center)},scale=\scaleValue,transform shape]
\node [circle,draw,fill] (v2) at (0.5,1) {};
\node [circle,draw,fill] (v1) at (-1,0) {};
\node [circle,draw,fill] (v3) at (2,0) {};
\node [circle,draw,fill] (v5) at (-0.5,-1.5) {};
\node [circle,draw,fill] (v4) at (1.5,-1.5) {};
\node [circle,draw,fill] (v6) at (1,-0.5) {};
\draw  (v1) edge (v2);
\draw  (v2) edge (v3);
\draw  (v3) edge (v4);
\draw  (v5) edge (v4);
\draw  (v5) edge (v1);
\draw  (v4) edge (v1);
\draw  (v2) edge (v5);
\draw  (v5) edge (v6);
\draw  (v6) edge (v2);
\draw  (v6) edge (v3);
\end{tikzpicture}\\
\cdots &=& \cdots \\
\mfk{w}_{2n+1} &=& \begin{tikzpicture}[baseline={([yshift=-0.5ex]current bounding box.center)},scale=\scaleValue,transform shape]
\node [circle,draw,fill] (v2) at (0.5,0.5) {};
\node [circle,draw,fill] (v1) at (-1.5,-0.4) {};
\node [circle,draw,fill] (v8) at (-2,-2) {};
\node [circle,draw,fill] (v7) at (-1.5,-3.6) {};
\node [circle,draw,fill] (v6) at (0.5,-4.5) {};
\node [circle,draw,fill] (v3) at (2.5,-0.4) {};
\node  (v4) at (3,-2) {\huge $\cdots$};
\node [circle,draw,fill] (v5) at (2.5,-3.6) {};
\draw  (v1) edge (v2);
\draw  (v2) edge (v3);
\draw  (v3) edge (v4);
\draw  (v5) edge (v4);
\draw  (v5) edge (v6);
\draw  (v6) edge (v7);
\draw  (v7) edge (v8);
\draw  (v8) edge (v1);
\node [circle,draw,fill] (v9) at (0.5,-2) {};
\draw  (v2) edge (v9);
\draw  (v9) edge (v3);
\draw  (v9) edge (v4);
\draw  (v9) edge (v5);
\draw  (v6) edge (v9);
\draw  (v9) edge (v7);
\draw  (v8) edge (v9);
\draw  (v1) edge (v9);
\end{tikzpicture} + \sum_{p=4}^{2n-1}{\lambda_p \Gamma^p}
\ea

\end{section}

\begin{section}{Kontsevich graph complex from the operad of Lie algebras} \label{KonGraphComp}

There is a relation between the operad of graphs and the operad $\oLie{d}$.

\begin{lemma}

Let $I$ be the ideal generated by graphs having at least one internal edge. Then

\[
\oLieC{d}/I \cong \gra_d.
\]

\end{lemma}
\begin{proof}

The operad $\oLie{d}/I$ is generated by elements $\Gamma$ of the form

\[
\begin{tikzpicture}[baseline={([yshift=-0.5ex]current bounding box.center)},scale=\scaleValue,transform shape]
\node [circle,draw,minimum size = \minRadius] (v1) at (0,1.25) {};
\node [circle,draw,fill] (v2) at (-2,0) {};
\node [circle,draw,fill] (v5) at (2,0) {};
\node [circle,draw] (v3) at (-3,-1) {\huge $1$};
\node [circle,draw] (v4) at (-1,-1) {\huge $2$};
\node [circle,draw,minimum size = 26] (v6) at (1,-1) {}; 
\node [circle,draw] (v7) at (3,-1) {\huge $n$};
\node at (0,0) {\huge $\cdots$};
\draw  (v1) edge (v2);
\draw  (v2) edge (v3);
\draw  (v2) edge (v4);
\draw  (v1) edge (v5);
\draw  (v6) edge (v5);
\draw  (v5) edge (v7);               
\end{tikzpicture}.
\]

Consider a linear map

\ba{rccc}
\alpha: & \oLieC{d} & \lra & \gra_d \\
 & \Gamma & \lma & \alpha(\Gamma),
\ea

where $\alpha(\Gamma)$ is a graph in $\gra_d$ which by definition has labelled white vertices identical to the white vertices of $\Gamma$, while the edges between vertices in $\alpha(\Gamma)$ correspond to irreducible components in $\Gamma$. If $d$ is even the ordering of edges attached to the white output vertex of $\Gamma$ induces an ordering of the edges in $\alpha(\Gamma)$. For $d$ odd, an orientation of the edges in $\alpha(\Gamma)$ is given by the ordering of the output edges attached to the white input vertices,


\[
\begin{tikzpicture}[baseline={([yshift=-0.5ex]current bounding box.center)},scale=\scaleValue,transform shape]
\node (v0) [circle,draw,minimum size = \minRadius] at (0,1.25) {};
\node (v1) [circle,draw,fill] at (0,0) {};
\node (v2) [circle,draw] at (-1,-1) {\huge $i$};
\node (v3) [circle,draw] at (1,-1) {\huge $j$};
\draw (v0) edge (v1);
\draw (v1) edge (v2);
\draw (v1) edge (v3);
\end{tikzpicture} \lra
\begin{tikzpicture}[baseline={([yshift=-0.5ex]current bounding box.center)},scale=\scaleValue,transform shape]
\node (v0) [circle,draw] at (0,0) {\huge $i$};
\node (v1) [circle,draw] at (3,0) {\huge $j$};
\draw [-stealth] (v0) edge (v1);
\end{tikzpicture}.
\]


It is clear that $\alpha$ is a bijection on the generators and it remains to show that the map respects operadic compositions. If $\Gamma_1,\Gamma_2 \in \oLieC{d}$ have no internal edges, then $\Gamma_1 \, o_i \, \Gamma_2$ is a sum of graphs with either one internal edge (which do not appear in $\oLieC{d}/I$) or an additional irreducible component. Moreover the terms with no internal edge correspond exactly under the map $\alpha$ to the operadic compositions in $\gra_d$. 

\end{proof}


This map $\alpha$ induces a map $\oLieC{d} \ra \oLieC{d}/I \ra \gra_d$ and thus a map of the deformation complexes 
\[
F:\deform(\LieOp_d \ra \oLieC{d}) \ra \fgc_d.
\]

We can now state the main theorem:


\begin{thm}

The map $F$ is a quasi-isomorphism.

\end{thm}
\begin{proof}

Recall that the natural projection $\fgc_d \ra \gc_d^{\geq 2}=\gc_d^2 \oplus \gc_d$ is a quasi-isomorphism and thus it is enough to show that the map $C:=\deform(\LieOp_d \ra \oLieC{d})  \substack{f}{\ra} \gc_d^{\geq 2}$ is a quasi-isomorphism. 


We notice that the generators of $\oLieC{d}$ can be suitably represented as connected graphs having two types of vertices, the white vertices corresponding to the input vertices and the star vertices corresponding to the irreducible connected components. This graph is obtained by first erasing the output vertex and all attached edges. Then we contract the internal vertices of each irreducible component to a single vertex, denoted $\stars$, which is decorated by the i.c.c. it originates. Last, the star vertices corresponding to $\begin{tikzpicture}[baseline={([yshift=-0.5ex]current bounding box.center)},scale=0.35,transform shape]
\node (v0) [circle,draw] at (0,0.85) {};
\node (v1) [circle,draw,fill] at (0,0) {};
\node (v2) [circle,draw] at (-0.75,-0.75) {};
\node (v3) [circle,draw] at (0.75,-0.75) {};
\draw (v0) edge (v1);
\draw (v1) edge (v2);
\draw (v1) edge (v3);
\end{tikzpicture}$ is replaced by an edge between the associated white vertices.

\begin{ex}
\item
\[
\begin{tikzpicture}[baseline={([yshift=-0.5ex]current bounding box.center)},scale=0.75,transform shape]
\node [circle,draw] (v2) at (0,1.25) {};
\node [circle,draw,fill] (v1) at (-0.5,0) {};
\node [circle,draw,fill] (v3) at (0.5,0) {};
\node [circle,draw] (v5) at (0,-1.5) {\Large $2$};
\node [circle,draw] (v6) at (-1.5,-1.5) {\Large $1$};
\node [circle,draw] (v4) at (1,-1.5) {\Large $3$};
\draw  (v1) edge (v2);
\draw  (v2) edge (v3);
\draw  (v3) edge (v4);
\draw  (v5) edge (v3);
\node [circle,draw,fill] (v7) at (-1.5,-0.5) {};
\draw [bend right] (v7) edge (v6);
\draw [bend left] (v7) edge (v6);
\draw  (v7) edge (v1);
\draw  (v1) edge (v5);
\end{tikzpicture}
\lra
\begin{tikzpicture}[baseline={([yshift=-0.5ex]current bounding box.center)},scale=0.75,transform shape]
\node [star,star points=10,draw] (v2) at (0,0) {};
\node [circle,draw] (v1) at (-1,-1) {\Large $1$};
\node [circle,draw] (v3) at (1,-1) {\Large $2$};
\node [circle,draw] (v4) at (2,-1) {\Large $3$};
\draw  (v1) edge [bend left] (v2);
\draw  (v1) edge [bend right] (v2);
\draw  (v2) edge (v3);
\draw  (v3) edge (v4);
\node at (0.85,0.2) {\icgDec{decorationHypGraph1}};
\end{tikzpicture}
\]

\item

\[
\begin{tikzpicture}[baseline={([yshift=-0.5ex]current bounding box.center)},scale=0.75,transform shape]
\node [circle,draw] (v2) at (0,1.25) {};
\node [circle,draw,fill] (v1) at (-0.5,0) {};
\node [circle,draw,fill] (v3) at (0.5,0) {};
\node [circle,draw] (v5) at (0,-1.5) {\Large $2$};
\node [circle,draw] (v6) at (-1.5,-1.5) {\Large $1$};
\node [circle,draw] (v4) at (1,-1.5) {\Large $3$};
\draw  (v1) edge (v2);
\draw  (v2) edge (v3);
\draw  (v3) edge (v4);
\draw  (v5) edge (v3);
\node [circle,draw,fill] (v7) at (-1,-0.5) {};
\draw  (v7) edge (v6);
\draw  (v7) edge (v1);
\draw  (v1) edge (v5);
\draw  (v7) edge (v5);
\end{tikzpicture}
\lra
\begin{tikzpicture}[baseline={([yshift=-0.5ex]current bounding box.center)},scale=0.75,transform shape]
\node [star,star points=10,draw] (v2) at (0,0) {};
\node [circle,draw] (v1) at (-1,-1) {\Large $1$};
\node [circle,draw] (v3) at (1,-1) {\Large $2$};
\node [circle,draw] (v4) at (2,-1) {\Large $3$};
\draw  (v1) edge (v2);
\draw  (v2) edge [bend left] (v3);
\draw  (v2) edge [bend right] (v3);
\draw  (v3) edge (v4);
\node at (1,0.4) {\icgDec{decorationHypGraph2}};
\end{tikzpicture}
\]

\item

\end{ex}

An edge between a white vertex and a star vertex is called \textit{star edge}. Note that by construction, there are no edges between star vertices.


The complex $C$ splits as a direct sum $C=(C^{\leq 1},\delta) \oplus (C^{\geq 2},\delta)$ where $C^{\leq 1}$ is generated by graphs with all white vertices having valency less or equal to $1$ and $C^{\geq 2}$ is generated by graphs with at least one white vertex of valency at least $2$.

The complex $C^{\leq 1}$ is acyclic. Indeed this complex is equal to

\[
(\text{span} \langle \begin{tikzpicture}[baseline={([yshift=-0.5ex]current bounding box.center)},scale=0.65,transform shape]
\node [circle,draw] at (0,0.85) {};
\node [circle,draw] at (0,-0.75) {};
\end{tikzpicture} ,\begin{tikzpicture}[baseline={([yshift=-0.5ex]current bounding box.center)},scale=0.65,transform shape]
\node (v0) [circle,draw] at (0,0.85) {};
\node (v1) [circle,draw,fill] at (0,0) {};
\node (v2) [circle,draw] at (-0.75,-0.75) {};
\node (v3) [circle,draw] at (0.75,-0.75) {};
\draw (v0) edge (v1);
\draw (v1) edge (v2);
\draw (v1) edge (v3);
\end{tikzpicture} , 
\begin{tikzpicture}[baseline={([yshift=-0.5ex]current bounding box.center)},scale=0.75,transform shape]
\node [star,star points=10,draw] (v1) at (0,0) {};
\node [circle,draw] (v2) at (0,1) {};
\node [circle,draw] (v3) at (1,0.5) {};
\node [circle,draw] (v4) at (1,-0.5) {};
\node [circle,draw] (v5) at (0,-1) {};
\node [circle,draw] (v6) at (-1,-0.5) {};
\node [circle,draw] (v7) at (-1,0.5) {};
\node at (-1,0.13)  {\huge \vdots};
\draw  (v1) edge (v2);
\draw  (v1) edge (v3);
\draw  (v1) edge (v4);
\draw  (v1) edge (v5);
\draw  (v6) edge (v1);
\draw  (v1) edge (v7);
\end{tikzpicture} \rangle , \delta_0).
\]

As the differential $\delta$ cannot create univalent white vertices we see that the complex where we omit the first graph is isomorphic to $\deform(\LieOp_d \ra \LieOp_d)$. As this complex is one dimensional with unique cohomology class given by $\begin{tikzpicture}[baseline={([yshift=-0.5ex]current bounding box.center)},scale=0.65,transform shape]
\node (v0) [circle,draw] at (0,0.85) {};
\node (v1) [circle,draw,fill] at (0,0) {};
\node (v2) [circle,draw] at (-0.75,-0.75) {};
\node (v3) [circle,draw] at (0.75,-0.75) {};
\draw (v0) edge (v1);
\draw (v1) edge (v2);
\draw (v1) edge (v3);
\end{tikzpicture}$, the equality $\delta \left(\begin{tikzpicture}[baseline={([yshift=-0.5ex]current bounding box.center)},scale=0.65,transform shape]
\node [circle,draw] at (0,0.5) {};
\node [circle,draw] at (0,-0.5) {};
\end{tikzpicture} \right) =\begin{tikzpicture}[baseline={([yshift=-0.5ex]current bounding box.center)},scale=0.65,transform shape]
\node (v0) [circle,draw] at (0,0.85) {};
\node (v1) [circle,draw,fill] at (0,0) {};
\node (v2) [circle,draw] at (-0.75,-0.75) {};
\node (v3) [circle,draw] at (0.75,-0.75) {};
\draw (v0) edge (v1);
\draw (v1) edge (v2);
\draw (v1) edge (v3);
\end{tikzpicture}$ and the above isomorphism imply that $C^{\leq 1}$ is acyclic. \\

Hence it is enough to prove that the restriction $f:(C^{\geq 2},\delta) \ra (\gc_d^2,0) \oplus (\gc_d,\delta)$ is a quasi-isomorphism. On the r.h.s. we consider a filtration on the number of vertices. Then on the first page the induced differential vanishes and the second page is equal to $(\gc_d^2,0) \oplus (\gc_d,\delta)$.

On the l.h.s. we consider a filtration defined by

\[
F_{-p}= \text{span of graphs with }\#\{\text{white vertices of valency at least $3$}\}+ 3\#\{\text{internal edges}\} \geq p.
\]

There is exactly one situation where the number of vertices of valency at least $3$ is decreasing. Let $v$ be a vertex of valency exactly $3$, e.g. consider

\[
\begin{tikzpicture}[baseline={([yshift=-0.5ex]current bounding box.center)},scale=0.65,transform shape]
\node (v0) [circle,draw] at (-1,-1) {};
\node (v1) [circle,draw,fill] at (0,0) {};
\draw (v0) edge (v1);
\node (v3) [circle,draw] at (1,-1) {\huge $v$};
\node (v4) [circle,draw,fill] at (2,0) {};
\node (v5) [circle,draw] at (3,-1) {};
\draw (v4) edge (v5);
\node (v6) [circle,draw,fill] at (2,1) {};
\node (v7) [circle,draw] at (1,2) {};
\draw (v4) edge (v6);
\draw (v6) edge (v7);
\draw (v1) edge (v7);
\draw (v3) edge (v1);
\draw (v6) [bend right] edge (v3);
\draw (v3) edge (v4);
\end{tikzpicture}.
\]

Then the differential can split $v$ in to two bivalent vertices $u$ and $w$. In the above example the differential creates (up to changing $u$ and $w$) three such terms

\[
\begin{tikzpicture}[baseline={([yshift=-0.5ex]current bounding box.center)},scale=0.65,transform shape]
\node (v0) [circle,draw] at (-1,-1) {};
\node (v1) [circle,draw,fill] at (0,0) {};
\draw (v0) edge (v1);
\node (v4) [circle,draw,fill] at (2,0) {};
\node (v5) [circle,draw] at (3,-1) {};
\draw (v4) edge (v5);
\node (v6) [circle,draw,fill] at (2,1) {};
\node (v7) [circle,draw] at (1,2) {};
\draw (v4) edge (v6);
\draw (v6) edge (v7);
\draw (v1) edge (v7);
\node (v8) [circle,draw,fill] at (1,-0.5) {};
\node (v9) [circle,draw] at (0.5,-1.5) {\huge $u$};
\node (v10) [circle,draw] at (1.5,-1.5) {\huge $w$};
\draw (v8) edge (v9);
\draw (v8) edge (v10);
\draw (v1) edge (v8);
\draw (v6) [bend right] edge (v9);
\draw (v4) edge (v10);
\end{tikzpicture}, \, \,
\begin{tikzpicture}[baseline={([yshift=-0.5ex]current bounding box.center)},scale=0.65,transform shape]
\node (v0) [circle,draw] at (-1,-1) {};
\node (v1) [circle,draw,fill] at (0,0) {};
\draw (v0) edge (v1);
\node (v4) [circle,draw,fill] at (2,0) {};
\node (v5) [circle,draw] at (3,-1) {};
\draw (v4) edge (v5);
\node (v6) [circle,draw,fill] at (2,1) {};
\node (v7) [circle,draw] at (1,2) {};
\draw (v4) edge (v6);
\draw (v6) edge (v7);
\draw (v1) edge (v7);
\node (v8) [circle,draw,fill] at (1,-0.5) {};
\node (v9) [circle,draw] at (0.5,-1.5) {\huge $u$};
\node (v10) [circle,draw] at (1.5,-1.5) {\huge $w$};
\draw (v8) edge (v9);
\draw (v8) edge (v10);
\draw (v1) edge (v9);
\draw (v6) [bend right] edge (v10);
\draw (v4) edge (v8);
\end{tikzpicture}, \, \,\,
\begin{tikzpicture}[baseline={([yshift=-0.5ex]current bounding box.center)},scale=0.65,transform shape]
\node (v0) [circle,draw] at (-1,-1) {};
\node (v1) [circle,draw,fill] at (0,0) {};
\draw (v0) edge (v1);
\node (v4) [circle,draw,fill] at (2,0) {};
\node (v5) [circle,draw] at (3,-1) {};
\draw (v4) edge (v5);
\node (v6) [circle,draw,fill] at (2,1) {};
\node (v7) [circle,draw] at (1,2) {};
\draw (v4) edge (v6);
\draw (v6) edge (v7);
\draw (v1) edge (v7);
\node (v8) [circle,draw,fill] at (1,-0.5) {};
\node (v9) [circle,draw] at (0.5,-1.5) {\huge $u$};
\node (v10) [circle,draw] at (1.5,-1.5) {\huge $w$};
\draw (v8) edge (v9);
\draw (v8) edge (v10);
\draw (v1) edge (v9);
\draw (v6) [bend right] edge (v8);
\draw (v4) edge (v10);
\end{tikzpicture}.
\]

Hence the number of white vertices of valency at least $3$ can decrease by $1$ but note that an internal edge needs to be created. In this case the number $F_{-p}$ increases by $1$ and hence is respected by the differential.

As the differential does not create new univalent vertices we see that the differential $\delta_0$ on the first page can only create new white bivalent vertices. As the number of star vertices is preserved, we have a direct sum decomposition $\text{gr}C^{\geq 2}=\bigoplus_{N \geq 0}{C_N}$ where $C_N$ is the subcomplex spanned by graphs with $N$ star vertices. It has been proven \cite{MR3348138} $H(C_0,\delta_0)=\gc_d^2 \oplus \gc_d$. It remains to show that $(\bigoplus_{N \geq 1}{C_N},\delta_0)$ is acyclic.


The differential $\delta_0$ can be represented by

\[
\delta_0(\Gamma) = -2\sum_{v \in V_o(\Gamma)}{\begin{tikzpicture}[baseline={([yshift=-0.5ex]current bounding box.center)},scale=\scaleValue,transform shape]
\node (v1) [circle,draw] at (0,0.5) {};
\node (v2) [circle,draw] at (0,-1) {\huge $\Gamma$};
\draw (v1) edge (v2);
\end{tikzpicture}}+\sum_{v \in V_o(\Gamma)}{\Gamma \, \text{o}_v \, \begin{tikzpicture}[baseline={([yshift=-0.5ex]current bounding box.center)},scale=\scaleValue,transform shape]
\node (v1) [circle,draw] at (0,0) {};
\node (v2) [circle,draw] at (1,0) {};
\draw (v1) edge (v2);
\end{tikzpicture}}\]

where $V_o(\Gamma)$ denotes the set of white vertices of $\Gamma$. The first part is given by summing over all attachments of a new univalent white vertex to a white vertex of $\Gamma$. The term $\Gamma \, o_v' \, \begin{tikzpicture}[baseline={([yshift=-0.5ex]current bounding box.center)},scale=\scaleValue,transform shape]
\node (v1) [circle,draw] at (0,0) {};
\node (v2) [circle,draw] at (1,0) {};
\draw (v1) edge (v2);
\end{tikzpicture}$ is given by splitting the white vertex $v$ and summing over all attachments of the halfedges pf $v$ to the two newly created vertices. As the graphs $\Gamma$ are connected this differential is equal to

\[
\delta_0(\Gamma)=\sum_{\stackrel{v \in V_o(\Gamma)}{\text{valency of $v \geq 2$}}}{\Gamma \, \text{o}_v \, \begin{tikzpicture}[baseline={([yshift=-0.5ex]current bounding box.center)},scale=\scaleValue,transform shape]
\node (v1) [circle,draw] at (0,0) {};
\node (v2) [circle,draw] at (1,0) {};
\draw (v1) edge (v2);
\end{tikzpicture}},
\]

where we omit the summands creating a univalent vertex.

We show that the complex spanned by graphs having at least one star vertex and at least one white vertex with valency at least $2$ is acyclic. It is enough to show that the complex spanned by such graphs with a labelling of star edges with integers is acyclic. In particular the set of star edges is totally ordered. Consider the star edge with minimum label which is attached to a white vertex of valency at least $2$. We call \textit{antenna} the sequence of white vertices starting from $\begin{tikzpicture}[baseline={([yshift=-0.5ex]current bounding box.center)},scale=0.75,transform shape]
\node [star,star points = 10,draw] (v1) at (-1,0) {};
\node [circle,draw] (v2) at (0.5,0) {};
\draw (v1) edge node [midway,above] {$i_{\min}$} (v2);
\end{tikzpicture}$ and ending with a white vertex of valency different from $2$ or a star vertex. The two-valent white vertices of this sequence are called \textit{antenna vertices}. Consider a filtration on the number of non-antenna vertices. To finish the argument it is sufficient to observe that the graphs with no antenna vertices and graphs with an antenna given by

\[
\begin{tikzpicture}[baseline={([yshift=-0.5ex]current bounding box.center)},scale=0.75,transform shape]
\node (v1) [star,star points=10,draw] at (0,0) {};
\node (v2) [circle,draw] at (1.5,0) {};
\node (v3) [star,star points=10,draw] at (3,0) {};
\draw (v1) edge node [midway,above] {$i_{\min}$} (v2);
\draw (v2) edge (v3);
\end{tikzpicture}
\]

are not cocycles and we conclude that $(\bigoplus_{N \geq 1}{C_N},\delta_0)$ is acyclic.


On the next page of the spectral sequence, there is a map

\[
(H^\ast(C_0,\delta_0),\delta_1) \lra (\gc_d^2,0) \oplus (\gc_d,\delta),
\]

where $H^\ast(C_0\delta_0) \cong \gc_d^2 \oplus \gc_d$ as vector spaces and $\delta_1$ is the differential which increases the value of the parameter of the filtration by $1$. Hence $\delta_1$ cannot create new star vertices (as this gives an increase by at least $2$) and thus $\delta_1$ can only split white vertices as the usual differential in $\gc_d$. Finally, this shows that we have an isomorphism of complexes on the second page and the result follows.

\end{proof}

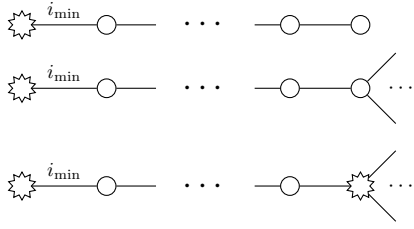
\begin{figure}

\begin{equation*}
\begin{array}{l}
\begin{tikzpicture}[baseline={([yshift=-0.5ex]current bounding box.center)},scale=0.75,transform shape]
\node [star,star points = 10,draw] (v1) at (-1,0) {};
\node [circle,draw] (v2) at (0.5,0) {};
\node [circle,draw] (v3) at (3.75,0) {};
\node [circle,draw] (v4) at (5,0) {};
\draw (v1) edge node [midway,above] {$i_{\min}$} (v2);
\draw  (v3) edge (v4);
\node (v5) at (1.5,0) {};
\node (v6) at (3,0) {};
\draw  (v2) edge (v5);
\draw  (v6) edge (v3);
\node at (2.25,0) {\huge $\cdots$};
\end{tikzpicture} \\
\begin{tikzpicture}[baseline={([yshift=-0.5ex]current bounding box.center)},scale=0.75,transform shape]
\node [star,star points = 10,draw] (v1) at (-1,0) {};
\node [circle,draw] (v2) at (0.5,0) {};
\node [circle,draw] (v3) at (3.75,0) {};
\node [circle,draw] (v4) at (5,0) {};
\draw (v1) edge node [midway,above] {$i_{\min}$} (v2);
\draw  (v3) edge (v4);
\node (v5) at (1.5,0) {};
\node (v6) at (3,0) {};
\draw  (v2) edge (v5);
\draw  (v6) edge (v3);
\node at (2.25,0) {\huge $\cdots$};
\node (v7) at (5.75,0.75) {};
\node (v8) at (5.75,-0.75) {};
\draw  (v4) edge (v7);
\draw  (v4) edge (v8);
\node at (5.75,0) {\large  $\cdots$};
\end{tikzpicture} \\
\begin{tikzpicture}[baseline={([yshift=-0.5ex]current bounding box.center)},scale=0.75,transform shape]
\node [star,star points = 10,draw] (v1) at (-1,0) {};
\node [circle,draw] (v2) at (0.5,0) {};
\node [circle,draw] (v3) at (3.75,0) {};
\node [star,star points=10,draw] (v4) at (5,0) {};
\draw (v1) edge node [midway,above] {$i_{\min}$} (v2);
\draw  (v3) edge (v4);
\node (v5) at (1.5,0) {};
\node (v6) at (3,0) {};
\draw  (v2) edge (v5);
\draw  (v6) edge (v3);
\node at (2.25,0) {\huge $\cdots$};
\node (v7) at (5.75,0.75) {};
\node (v8) at (5.75,-0.75) {};
\draw  (v4) edge (v7);
\draw  (v4) edge (v8);
\node at (5.75,0) {\large  $\cdots$};
\end{tikzpicture}
\end{array}
\end{equation*}

\caption{The different types of antennas.}

\end{figure}

As a corollary we obtain

\begin{cor}

\[
H^\bullet(\deform(\LieOp_d \stackrel{i}{\ra} \oLieC{d}))= \gc_d^2 \oplus H^\bullet(\gc_d)
\]


In particular


\[
H^0(\defOLie{2})=\grt_1.
\]

\end{cor}

Another application is the following: we know that $\deform(\LieOp_d \stackrel{i}{\ra} \oLieC{d})$ controls the infinitesimal homotopy non-trivial deformations of $i$, which are given by $H^1(\deform(\LieOp_d \stackrel{i}{\ra} \oLieC{d})) \cong H^1(\gc_d^2) \oplus H^1(\gc_d)$, while the obstructions to extending an infinitesimal deformation $\Delta$ of $i$ to a genuine morphism of (completed) operads

\[
i^\Delta:\hoLie_d \ra \oLieC{d}
\]

lie in $H^2(\gc_d^2) \oplus H^2(\gc_d)$. \\

We now consider the case $d=1$ corresponding to the operad of usual Lie algebras. It is noticed in \cite{MR3312446} that $H^1(\gc_1^2)=0$ while $H^1(\gc_1)$ is one-dimensional and is generated by the theta graph

\[
\Delta:=\begin{tikzpicture}
[baseline={([yshift=-0.5ex]current bounding box.center)},scale=0.5,transform shape]
\node (v0) [circle,draw,fill,label=below: \huge $1$] at (0,0) {};
\node (v1) [circle,draw,fill,label=below: \huge $2$] at (2,0) {};
\draw [bend left,-latex] (v0) edge (v1);
\draw [bend right,-latex] (v0) edge (v1);
\draw [-latex] (v0) edge (v1);
\end{tikzpicture}.
\]

Therefore the standard morphism $i:\LieOp \ra \oLieC{}$ admits precisely one homotopy non-trivial infinitesimal deformation corresponding to the above mentioned theta graph, which in our approach is incarnated as the following element in $\oLieC{}$:

\[
\begin{tikzpicture}
[baseline={([yshift=-0.5ex]current bounding box.center)},scale=0.5,transform shape]
\node (v0) [circle,draw,fill,label=below: \huge $1$] at (0,0) {};
\node (v1) [circle,draw,fill,label=below: \huge $2$] at (2,0) {};
\draw [bend left,-latex] (v0) edge (v1);
\draw [bend right,-latex] (v0) edge (v1);
\draw [-latex] (v0) edge (v1);
\end{tikzpicture} \cong \begin{tikzpicture}[baseline={([yshift=-0.5ex]current bounding box.center)},scale=\scaleValue,transform shape]
\node (v0) [circle,draw,minimum size = \minRadius] at (0,1.25) {};
\node (v1) [circle,draw,fill] at (0,0) {};
\node (v2) [circle,draw,fill] at (-1,0) {};
\node (v3) [circle,draw,fill] at (1,0) {};
\node (v4) [circle,draw] at (-1.5,-1) {\huge $1$};
\node (v5) [circle,draw] at (1.5,-1) {\huge $2$};
\draw (v0) edge (v1);
\draw (v0) edge (v2);
\draw (v0) edge (v3);
\draw (v1) edge (v4);
\draw (v1) edge (v5);
\draw (v2) edge (v4);
\draw (v2) edge (v5);
\draw (v3) edge (v4);
\draw (v3) edge (v5);
\end{tikzpicture}.
\]

Moreover the second cohomology group is generated by 

\[
\begin{tikzpicture}[baseline={([yshift=-0.5ex]current bounding box.center)},scale=\scaleValue,transform shape]
\node (v0) [circle,draw,fill,label=below: \huge $1$] at (0,0) {};
\node (v1) [circle,draw,fill,label=below: \huge $2$] at (2,0) {};
\node (v2) [circle,draw,fill,label=above: \huge $3$] at (1,1.7) {};
\draw [-latex] (v0) edge (v1);
\draw [-latex] (v0) edge (v2);
\draw [-latex] (v1) edge (v2); 
\end{tikzpicture}.
\]

This obstruction cohomology class cannot be hit when we try to extend the infinitesimal deformation $\Delta$ to a genuine deformation (cf. \cite{MR3312446} Section $5$). Hence we conclude that the standard morphism $i:\LieOp \ra \oLieC{}$ admits precisely one (up to homotopy equivalence) non-trivial deformation. Remarkably, this unique non-trivial deformation is directly related to the universal enveloping construction and the PBW theorem as we show in Section \ref{uniqueDeformation}.

\end{section}

\begin{section}{Gutt quantizations} \label{guttQuant}

Let $V$ be a vector space equipped with a Lie bracket $[,]$. This is equivalent to a morphism of operads


\[
\rho:\LieOp \lra \edom_V.
\] 


Using the functor $\oFun$, we see that there is an associated morphism of operads


\[
\hat{\rho}: \oLie{} \lra \edom_{\odot V}
\]


\noi given by poly-differential operators.


The space $\odot V$ can be canonically given a structure of an associative algebra, with product denoted by $\ast$, as follows: denote by $\mcl{U}V$ the universal enveloping algebra of $V$, i.e.


\[
\mcl{U}V:=\mcl{T}V/\langle v_1 \otimes v_2 - v_2 \otimes v_1 - [v_1,v_2] | v_1,v_2 \in V \rangle,
\]


\noi where $\mcl{T}V:=\bigoplus_{n \geq 0}{\bigotimes^nV}$ is the tensor algebra. Then $\mcl{U}V$ is an associative algebra with product $\odot$ given by $\otimes$ mod $I$.


The spaces $\odot V$ and $\mcl{U}V$ are related as vector spaces by the Poincaré-Birkhoff-Witt (PBW) theorem:

\begin{thm}

As a vector space $\odot V \cong \mcl{U}V$.

\end{thm}

More precisely consider the linear map \cite{vinayKontsevich}


\ba{rccc}
\sigma: & \odot V & \lra & \mcl{U}V \\
& v_1 \odot \cdots \odot v_n & \lma & \sum_{\tau \in \mbb{S}_n}{\frac{1}{n!}v_{\tau(1)}\otimes \cdots \otimes v_{\tau(n)}}.
\ea


By PBW this is an isomorphism of vector spaces and this allows us to define a product of $p,q \in \odot V$ by


\[
p \ast q := \sigma^{-1}(\sigma(p) \circ \sigma(q)).
\]


It has been shown \cite{vinayKontsevich,MR2795152} that 


\begin{equation}
p\ast q =p\odot q + \sum_{m,n \geq 1}{B_{m,n}(p,q)},
\label{eqnStar}
\end{equation}


\noi where $B_{m,n}$ are bi-differential operators of order $m$ on $p$ and order $n$ on $q$. The product $\ast$ is called $\mcl{U}$-star product.


If $p$ and $q$ are polynomials, then the r.h.s is a finite sum. If we try to extend $(\ref{eqnStar})$ to general smooth functions on $v$, then the sum can be infinite and thus we run into convergence issues. This can be solved by introducing a parameter $\hbar$ and by working with


\[
\mcl{U}_\hbar V:=\mcl{T}V[[\hbar]]\langle v_1 \otimes v_2 - v_2 \otimes v_1 - \hbar[v_1,v_2] | v_1,v_2 \in V \rangle.
\]


\begin{ex}

\begin{enumerate}[label=\arabic*)]

\item If $v_1,v_2 \in V$, then


\begin{equation}
v_1 \ast v_2 = v_1 \odot v_2 + \frac{\hbar}{2}[v_1,v_2].
\label{eqnTrivial}
\end{equation}

\item Let $e^{v_1},e^{v_2}\in C^\infty_V$. Then


\[
e^{v_1} \ast e^{v_2}=e^{\text{CBH}(v_1,v_2)},
\]


\noi where $\text{CBH}(v_1,v_2)$ denotes the Campbell-Baker-Hausdorff series


\[
\text{CBH}(v_1,v_2)=v_1+v_2+\frac{\hbar}{2}[v_1,v_2] + \frac{\hbar^2}{12}([v_1,[v_1,v_2]]+[[v_1,v_2],v_2])+\cdots.
\]

\end{enumerate}

\end{ex}

We conclude that the $\mcl{U}$-star product gives us a map $\mcl{U}_V:\assOp \ra \edom_{\odot V}$.
As this map does not depend on the choice of $V$ and is given by poly-differential operators, we infer that it factors through $\oLie{}$, i.e.


\[
\begin{tikzcd}
\assOp \arrow{r}{\mcl{U}} \arrow[dashed,swap]{dr}{\mcl{U}_V} & \oLie{} \arrow{d}{\hat{\rho}} \\
& \edom_{\odot V}
\end{tikzcd},
\]


\noi for some morphism $\mcl{U}$ of operads satisfying the non-triviality condition


\begin{equation}
\begin{tikzpicture}[baseline={([yshift=-0.5ex]current bounding box.center)},scale=\scaleValue,transform shape]
\node (v0) at (0,1) {};
\node (v1) [circle,draw] at (0,0) {};
\node (v2) at (-1,-1) {\huge $1$};
\node (v3) at (1,-1) {\huge $2$};
\draw (v0) edge (v1);
\draw (v1) edge (v2);
\draw (v1) edge (v3);
\end{tikzpicture}
\lma
\begin{tikzpicture}[baseline={([yshift=-0.5ex]current bounding box.center)},scale=\scaleValue,transform shape]
\node [circle,draw,minimum size = \minRadius] at (0,1.25) {};
\node [circle,draw] at (-1,-1) {\huge $1$};
\node [circle,draw] at (1,-1) {\huge $2$};
\end{tikzpicture}
+\frac{1}{2}
\begin{tikzpicture}[baseline={([yshift=-0.5ex]current bounding box.center)},scale=\scaleValue,transform shape]
\node (v0) [circle,draw,minimum size = \minRadius] at (0,1.25) {};
\node (v1) [circle,draw,fill] at (0,0) {};
\node (v2) [circle,draw] at (-1,-1) {\huge $1$};
\node (v3) [circle,draw] at (1,-1) {\huge $2$};
\draw (v0) edge (v1);
\draw (v1) edge (v2);
\draw (v1) edge (v3);
\end{tikzpicture}
+
\text{terms with at least two internal vertices}.
\label{eqnTrivialOperads}
\end{equation}


This motivates the following definition:

\begin{defi}

Any morphism of operads $\mcl{U}:\assOp \ra \oLie{}$ satisfying $(\ref{eqnTrivialOperads})$ is called a \textit{S. Gutt quantization}.

\end{defi}

That the standard universal enveloping construction gives us an example of such a quantization was first noticed by Simone Gutt in \cite{MR2795152}. We classify below all such quantizations up to homotopy equivalence, and prove that they are all gauge equivalent to the universal enveloping one.

\begin{thm}
\label{theorem:gutt_one_dimensional}

Let $U$ be a S. Gutt quantization. Then the cohomology of $\deform(\assOp \ra \oLie{})$ is one-dimensional.

\end{thm}
\begin{proof}

Recall the deformation complex


\[ \deform(\assOp \stackrel{U}{\ra} \oLie{})=\prod_{n \geq 1}{\assOp(n)^\ast \otimes (\oLie{}(n))^{\mbb{S}_n}[-1]}=\prod_{n \geq 1}(\oLie{}(n))^{\mbb{S}_n}[1-n]
\]


\noi  equipped with the differential $d$ associated to $U$.


Consider a filtration by the number of internal vertices. The differential $d$ clearly respects the filtration and we consider the associated spectral sequence $(\mcl{E}_n,d_n)_{n\geq 0}$. The differential on the first page is only induced by the first term of $U$. We observe that if $\Gamma$ is connected, then the surviving terms are also connected and the differential is given by


\[
d_0 \Gamma = \sum_i \delta_i(\Gamma),
\]


\noi where $i$ ranges over the number of white vertices and $\delta_i$ acts by


\[
\delta_i \begin{tikzpicture}[baseline={([yshift=-0.5ex]current bounding box.center)},scale=\scaleValue,transform shape]
\node [circle,draw] (v1) at (0,0) {\huge $i$};
\node (v2) at (-1,1) {};
\node at (0.05,0.7) {\huge $\cdots$};
\node (v3) at (1,1) {};
\draw  (v1) edge (v2);
\draw  (v1) edge (v3);
\draw [decorate,decoration={brace,amplitude=5 pt},thick] (v2) -- (v3);
\node at (0,1.9) {\huge $I$};
\end{tikzpicture}=\sum_{\stackrel{I=I_1 \sqcup I_2}{|I_1|,|I_2| \geq 1}}{\begin{tikzpicture}[baseline={([yshift=-0.5ex]current bounding box.center)},scale=\scaleValue,transform shape]
\node [circle,draw] (v1) at (0,0) {\huge $i_1$};
\node (v2) at (-1.05,1.3) {};
\node at (0,1) {\huge $\cdots$};
\node (v3) at (0.95,1.3) {};
\draw  (v1) edge (v2);
\draw  (v1) edge (v3);
\node [circle,draw] (v4) at (3,0) {\huge $i_2$};
\node (v5) at (1.95,1.3) {};
\node (v6) at (3.95,1.3) {};
\node (v7) at (3,1) {\huge $\cdots$};
\draw [decorate,decoration={brace,amplitude=5 pt},thick] (v5) -- (v6);
\draw (v4) edge (v5);
\draw (v4) edge (v6);
\draw [decorate,decoration={brace,amplitude=5 pt},thick] (v2) -- (v3);
\node at (-0.05,2.1) {\huge $I_1$};
\node at (3,2.1) {\huge $I_2$};
\end{tikzpicture}},
\] 


This follows from the fact that the two terms with a disconnected white vertex only differ swapping $i_1$ and $i_2$ and thus by the skew-symmetrization of labels of white vertices these terms are opposites. In particular $\delta_i$ and $d_0(\Gamma)=0$ if all white vertices are univalent.


In \cite{MR2762550} it has been shown that the cohomology on the first page $H(\mcl{E}_0,d_0)$ is equal to the span of graphs of $\oLie{}$ such that every white vertex is univalent and the labelled white vertices are skew-symmetrized, i.e.


\[
H(\mcl{E}_0,d_0)=\langle\begin{tikzpicture}[baseline={([yshift=-0.5ex]current bounding box.center)},scale=\scaleValue,transform shape]
\node [circle,draw,minimum size = \minRadius] (v1) at (0,1.25) {};
\node [circle,draw,fill] (v2) at (-1,0) {};
\node [circle,draw,minimum size = \minRadius] (v3) at (0.5,-2) {};
\node [circle,draw,fill] (v4) at (-2,-1) {};
\node [circle,draw,minimum size = \minRadius] (v5) at (-3,-2) {};
\node [circle,draw,minimum size = \minRadius] (v6) at (-1,-2) {};
\node [circle,draw,fill] (v7) at (3,-0.5) {};
\node [circle,draw,minimum size = \minRadius] (v8) at (2,-2) {};
\draw  (v1) edge (v2);
\draw  (v2) edge (v3);
\draw  (v2) edge (v4);
\draw  (v4) edge (v5);
\draw  (v4) edge (v6);;
\draw  (v7) edge (v8);
\draw  (v1) edge (v7);
\node [circle,draw,minimum size = \minRadius] (v9) at (4,-2) {};
\draw  (v7) edge (v9);
\end{tikzpicture} \rangle.
\]


This complex can be identified with


\[
\text{span} \left\langle \begin{tikzpicture}[baseline={([yshift=-0.5ex]current bounding box.center)},scale=\scaleValue,transform shape]
\node  (v1) at (-1,1) {};
\node [circle,draw,fill] (v2) at (-1,0) {};
\node  (v3) at (0.5,-2) {};
\node [circle,draw,fill] (v4) at (-2,-1) {};
\node  (v5) at (-3,-2) {};
\node  (v6) at (-1,-2) {};
\node [circle,draw,fill] (v7) at (2.5,-0.5) {};
\node (v8) at (1.5,-2) {};
\draw  (v1) edge (v2);
\draw  (v2) edge (v3);
\draw  (v2) edge (v4);
\draw  (v4) edge (v5);
\draw  (v4) edge (v6);;
\draw  (v7) edge (v8);
\node  (v9) at (3.5,-2) {};
\draw  (v7) edge (v9);
\node (v10) at (2.5,1) {};
\draw  (v7) edge (v10);
\end{tikzpicture} 
 \right\rangle
\]


\noi where the outputs are symmetrized and the inputs are skew-symmetrized. This space can be seen as a subspace of $\odot (\deform(\LieOp \ra \LieOp)[-1])[1]$. Since we are working over a field we see that the cohomology of this space is one-dimensional as the cohomology of $\deform(\LieOp \ra \LieOp)$ is concentrated in degree $1$. 

\end{proof}

\begin{cor}

Let $U$ be a S. Gutt quantization. Then the map


\[
\deform(\LieOp \ra \LieOp) \lra \deform(\assOp \stackrel{U}{\ra} \oLie{}) 
\]


\noi is a quasi-isomorphism.

\end{cor}

We will now generalize the above results to $\LieOp_\infty$ and $\hoAss$ structures.

Let $V$ be a $\mbb{Z}$-graded vector space.

\begin{defi}

A homotopy S. Gutt quantization is a morphism of dg operads


\[
\varphi : \hoAss \lra \oFun(\LieOp_\infty)
\]


satisfying

\begin{equation*}
\begin{tikzpicture}[baseline={([yshift=-0.5ex]current bounding box.center)},scale=\scaleValue,transform shape]
\node [circle,draw] (v1) at (0,0) {};
\node (v2) at (-1,-1.5) {\huge $1$};
\node (v3) at (1,-1.5) {\huge $n$};
\node at (0.055,-1) {\huge $\cdots$};
\draw  (v1) edge (v2);
\draw  (v1) edge (v3);
\node (v4) at (0,1.25) {};
\draw  (v1) edge (v4);
\end{tikzpicture}
\lma
\left\{ \begin{array}{lc}
\begin{tikzpicture}[baseline={([yshift=-0.5ex]current bounding box.center)},scale=\scaleValue,transform shape]
\node [circle,draw,minimum size = \minRadius] at (0,1.25) {};
\node [circle,draw] at (-1,-1) {\huge $1$};
\node [circle,draw] at (1,-1) {\huge $2$};
\end{tikzpicture}
+\frac{1}{2}
\begin{tikzpicture}[baseline={([yshift=-0.5ex]current bounding box.center)},scale=\scaleValue,transform shape]
\node (v0) [circle,draw,minimum size = \minRadius] at (0,1.25) {};
\node (v1) [circle,draw,fill] at (0,0) {};
\node (v2) [circle,draw] at (-1,-1) {\huge $1$};
\node (v3) [circle,draw] at (1,-1) {\huge $2$};
\draw (v0) edge (v1);
\draw (v1) edge (v2);
\draw (v1) edge (v3);
\end{tikzpicture}
+
\text{terms with at least two internal vertices}.

& \text{if $n=2$,} \\[.5 cm ]
\frac{1}{n!}\begin{tikzpicture}[baseline={([yshift=-0.5ex]current bounding box.center)},scale=\scaleValue,transform shape]
\node [circle,draw,fill] (v1) at (0,0) {};
\node [circle,draw] (v2) at (-1,-1.5) {\huge $1$};
\node [circle,draw] (v3) at (1,-1.5) {\huge $n$};
\node at (0.055,-1) {\huge $\cdots$};
\node [circle,draw,minimum size = \minRadius] (v4) at (0,1.25) {};
\draw  (v1) edge (v2);
\draw  (v1) edge (v3);
\draw  (v1) edge (v4);
\end{tikzpicture}
+
\text{terms with at least two internal vertices} & \text{if $n \geq 3$.}
\end{array} \right.
\end{equation*}

\end{defi}

\begin{ex}

\begin{enumerate}[label=\arabic*)]

\item The constructions given by Baranovsky \cite{MR2470385}, Moreno-Fernández \cite{MR4381197} and Lada and Markl \cite{MR1327129} can be interpreted as such a map.

\item The Kontsevich formality map \cite{MR2062626} implies this formula, but as a map to a wheeled version $\oFun(\LieOp_\infty^\circlearrowright)$ which is ill-defined for infinite-dimensional Lie algebras. This was resolved by Shoikhet in \cite{MR1854132} by showing that wheels can be set to zero and thus we get a map of the desired form. 

\end{enumerate}

\end{ex}

\begin{cor}
\label{cor:GuttInfinity}

Let $\varphi$ be a homotopy S. Gutt quantization. Then the complex $\deform(\hoAss \stackrel{\varphi}{\ra} \oFun(\LieOp_\infty))$ is one-dimensional. In particular the map


\[
\deform(\hoAss \stackrel{\varphi}{\ra} \oFun(\LieOp_\infty)) \cong \deform(\LieOp_\infty \ra \LieOp_\infty).
\]


\noi is a quasi-isomorphism.

\end{cor}
\begin{proof}

Since the canonical map $\pi :\hoLie_1 \ra \LieOp$ is a quasi-isomorphism and the functor $\oFun$ is exact, we see that the map $\oFun(\pi)$ is a quasi-isomorphism. Hence the deformation complexes $\deform(\hoAss \stackrel{\varphi}{\ra} \oFun(\hoLie_1))$ and $\deform(\hoAss \stackrel{\oFun(\pi)\varphi}{\ra} \oFun(\LieOp))$ are quasi-isomorphic. In addition the map $\oFun(\pi) \varphi: \hoAss \ra \oFun(\LieOp)$ factors through $\assOp$, i.e.


\[
\begin{tikzcd}
\hoAss \arrow{r}{\pi'} \arrow[dashed,swap]{dr}{\oFun(\pi)\varphi} & \assOp \arrow{d}{\mcl{U}} \\
& \oFun(\LieOp)
\end{tikzcd},
\]


\noi where $\mcl{U}$ is a S. Gutt quantization. Thus 


\[
H(\deform(\hoAss \stackrel{\oFun(\pi)\varphi}{\ra} \oFun(\LieOp)))=H(\deform(\assOp \stackrel{\mcl{U}}{\ra} \oFun(\LieOp))).
\]


By Theorem \ref{theorem:gutt_one_dimensional} the cohomology $H(\deform(\assOp \stackrel{\mcl{U}}{\ra} \oFun(\LieOp)))$ is one-dimensional and the result follows.

\end{proof}

This result fully agrees with the PBW theorem obtained in \cite{Khoroshkin2023} using  the bar-cobar duality in the homotopy theory.

\end{section}

\begin{section}{On the unique non-trivial deformation of the map $i$} \label{uniqueDeformation}

Let $\mfk{g}$ be a Lie algebra over some field $\mbb{K}$ with a countable basis $\{t_i\}_{i \in I}$. Then $\odot \mfk{g}$ can be identified with the polynomial ring $\mbb{K}[t_I]$. The Gutt quantization formula or the PBW quantization formula applied to polynomials $P(t)$ and $Q(t)$ can be given in terms of differential operators as follows (see Theorem 5 in \cite{bekaert2005universal}):

\[
P(t) \ast Q(t) = \exp(\sum_i{t_im^i(\frac{\partial}{\partial u}\frac{\partial}{\partial v})})\left.P(u)Q(v)\right|_{u=v=t},
\]

where $m^i$ is defined as follows: Let $X=\sum_i{x^it_i}$, $Y=\sum_i{y^it_i}$ be arbitrary elements of $\mfk{g}$ (we understand the numerical coefficients $x^i$ and $y^i$ as formal parameters). Then define

\[
m(X,Y):=\log(e^X e^Y)-X-Y = \sum{t_im^i_{j_1,\cdots,j_p,k_1,\cdots,k_q}x^{j_1}\cdots x^{j_p}y^{k_1} \cdots y^{k_q}}
\] 

and set $m(\frac{\partial}{\partial u},\frac{\partial}{\partial v})$ to be the differential operator obtained from the above power series by replacing each $x^{j_1}\cdots x^{j_p}$ by $\frac{\partial}{\partial u_{j_1}}\cdots \frac{\partial}{\partial u_{j_p}}$ and similarly for $y^{k_1}\cdots y^{k_q}$. It is hard in general to rewrite this formula in terms of our graphs as an explicit morphism of operads

\[
\assOp \lra \oLie{}
\]

but we only need to see its quotient modulo the graphs $I \subset \oLie{}$ containing at least one internal edge

\[
\assOp \lra \oLie{}/I
\]

which is given explicitly by

\[
\begin{tikzpicture}[baseline={([yshift=-0.5ex]current bounding box.center)},scale=\scaleValue,transform shape]
\node (v0) [circle,draw,fill,label=below: \huge $1$] at (0,0) {};
\node (v1) [circle,draw,fill,label=below: \huge $2$] at (2,0) {};
\end{tikzpicture}
+
\begin{tikzpicture}[baseline={([yshift=-0.5ex]current bounding box.center)},scale=\scaleValue,transform shape]
\node (v0) [circle,draw,fill,label=below: \huge $1$] at (0,0) {};
\node (v1) [circle,draw,fill,label=below: \huge $2$] at (2,0) {};
\draw [-latex] (v0) edge (v1);
\end{tikzpicture}
+\frac{1}{2}\begin{tikzpicture}[baseline={([yshift=-0.5ex]current bounding box.center)},scale=\scaleValue,transform shape]
\node (v0) [circle,draw,fill,label=below: \huge $1$] at (0,0) {};
\node (v1) [circle,draw,fill,label=below: \huge $2$] at (2,0) {};
\draw [bend left,-latex] (v0) edge (v1);
\draw [bend right,-latex] (v0) edge (v1);
\end{tikzpicture}
+\frac{1}{3!}\begin{tikzpicture}[baseline={([yshift=-0.5ex]current bounding box.center)},scale=\scaleValue,transform shape]
\node (v0) [circle,draw,fill,label=below: \huge $1$] at (0,0) {};
\node (v1) [circle,draw,fill,label=below: \huge $2$] at (2,0) {};
\draw [-latex] (v0) edge (v1);
\draw [bend left,-latex] (v0) edge (v1);
\draw [bend right,-latex] (v0) edge (v1);
\end{tikzpicture}
+\cdots .
\]

After skew-symmetrization of the indices $1$ and $2$ this series simplifies as follows

\[
\begin{tikzpicture}[baseline={([yshift=-0.5ex]current bounding box.center)},scale=\scaleValue,transform shape]
\node (v0) [circle,draw,fill] at (0,0) {};
\node (v1) [circle,draw,fill] at (2,0) {};
\draw [-latex] (v0) edge (v1);
\end{tikzpicture}
+\frac{1}{3!}\begin{tikzpicture}[baseline={([yshift=-0.5ex]current bounding box.center)},scale=\scaleValue,transform shape]
\node (v0) [circle,draw,fill] at (0,0) {};
\node (v1) [circle,draw,fill] at (2,0) {};
\draw [-latex] (v0) edge (v1);
\draw [bend left,-latex] (v0) edge (v1);
\draw [bend right,-latex] (v0) edge (v1);
\end{tikzpicture}
+\frac{1}{5!}\begin{tikzpicture}[baseline={([yshift=-0.5ex]current bounding box.center)},scale=\scaleValue,transform shape]
\node (v0) [circle,draw,fill] at (0,0) {};
\node (v1) [circle,draw,fill] at (2,0) {};
\draw [-latex] (v0) edge (v1);
\draw [bend left,-latex] (v0) edge (v1);
\draw [bend right,-latex] (v0) edge (v1);
\draw [bend left=75,-latex] (v0) edge (v1);
\draw [bend right=75,-latex] (v0) edge (v1);
\end{tikzpicture}
+\cdots .
\]

As we see the first non-trivial term is precisely our theta class, we complete the proof that the composition

\[
\LieOp \lra \assOp \lra \oLie{},
\]

is precisely the unique homotopy non-trivial deformation of the naive map $i$.

\end{section}


\nocite{*}
\bibliographystyle{alpha}
\bibliography{Article12022bib}

\end{document}